\documentclass[10pt,a4paper]{amsart}
\usepackage{amssymb,amsmath,amsfonts,amscd}
\usepackage{lmodern}
\usepackage[utf8]{inputenc} 
\usepackage[T1]{fontenc}    
\usepackage{hyperref}       
\usepackage{url}            
\usepackage{booktabs}       
\usepackage{fancyhdr}
\usepackage{indentfirst}
\usepackage{graphicx}
\usepackage{subfigure}
\usepackage{wrapfig}
\usepackage{newlfont}
\usepackage{color}
\usepackage{latexsym}
\usepackage{csquotes}
\usepackage{amsthm}
\usepackage{mathtools}
\usepackage{nicefrac}
\usepackage{epstopdf}
\usepackage{caption}
\usepackage{bbm}
\usepackage[shortlabels]{enumitem}
\usepackage{tikz-cd}
\usepackage{microtype}      
\usepackage{lipsum}
\graphicspath{ {./images/} }
\usepackage{comment}

\DeclareMathOperator*{\esssup}{ess\,sup}
\DeclareMathOperator*{\essinf}{ess\,inf}

\newcommand{\R}{\mathbb R}

\newcommand{\Z}{\mathbb Z}

\newcommand{\N}{\mathbb N}

\def \N {\mathbb{N}}
\def \R {\mathbb{R}}

\def \Z {\mathbb{Z}}

\def \d {\mathrm{d}}

\def \de {\partial}

\def \e {\varepsilon }

\def \dsy {\displaystyle}

\def \div {\mathrm{div}}

\newtheorem{teo}{Theorem}[section]
\newtheorem{prop}[teo]{Proposition}
\newtheorem{coro}[teo]{Corollary}
\newtheorem{lemma}[teo]{Lemma}
\newtheorem{defi}[teo]{Definition}

\newtheorem{remark}[teo]{Remark}

\newtheorem{oss}[teo]{Remark}

\newtheorem{hypothesis}[teo]{Hypothesis}

\title{Analysis of the\\ 
Quasi-Static Maxwell Equations 
\\in Resistive Solid-State Particle Detectors}

\keywords{Maxwell equations, mixed boundary conditions, evolution equations, Hille-Yosida operators}

\subjclass{35Q61, 35G15, 47D06, 46N20}

\author{
 Alessandro Rosa 
}

\begin{document}

\begin{abstract}
We solve a boundary value problem arising from Maxwell’s equations in the quasi-static approximation, which governs the time evolution of the so-called weighting potential $V_{\mathrm{w}}(t,x)$ in resistive solid-state particle detectors. The model reduces to the third-order time-dependent PDE
$$
\varepsilon\,\partial_t \Delta V_{\mathrm{w}}(t,x)
+ \operatorname{div}\!\big(\sigma(x)\nabla V_{\mathrm{w}}(t,x)\big)=0
\quad\text{in } [0,T]\times\Omega,
$$
supplemented with mixed Neumann-Dirichlet boundary conditions, possibly degenerate.  
Our analysis is based on the decomposition of the weighting potential into a static and a dynamic component. The static part solves a uniformly elliptic mixed boundary value problem, while the dynamic part satisfies a degenerate parabolic Cauchy problem.
We also establish interior regularity results.
\end{abstract}



\maketitle

\section{Introduction}

One of the main challenges in nuclear and sub-nuclear physics is the accurate detection of charged particles produced in high-energy collisions. These measurements are essential for testing the predictions of the Standard Model and for exploring possible physics beyond it. Major experimental programs addressing these questions are carried out at the CERN (European Council for Nuclear Research), in particular at the Large Hadron Collider (LHC), which is currently the largest particle accelerator in the world. The development of next-generation tracking detectors for the High-Luminosity LHC \cite{DaVia:2012ay}, and future particle colliders such as the FCC-hh \cite{FCC:2018vvp} and the Muon Collider \cite{InternationalMuonCollider:2024jyv} requires devices capable of operating under extremely high radiation levels while preserving excellent spatial and temporal resolution.

Among the technologies employed in high-energy physics experiments, \textit{resistive solid-state particle detectors} play a fundamental role in tracking ionizing particles produced in accelerator collisions. These devices typically consist of a semiconductor volume, commonly silicon, equipped with two or more electrodes placed either on the boundaries or embedded within the bulk material (in this last case, they are called 3D solid-state particle detectors).
In this context, diamond sensors represent promising technology due to some of their properties, such as high carrier mobility, radiation hardness and the possibility to realize many 3D geometry through laser graphitization of the bulk.
For these reasons, although our mathematical analysis applies to general resistive solid‑state detectors, we focus on the diamond-graphite configuration as a representative and technologically relevant case of study.

A charged particle traversing a material deposits part of its energy, enabling the creation of electron-hole pairs, also referred as charge carriers.
Under the action of the electric field generated by the voltage applied to the electrodes, charge carriers drift through the bulk along trajectories determined by the electric field configuration, reaching the electrodes.
The electrical signal generated in a detector with resistive elements by the motion of charge carriers can be described using the time-dependent version of the \textit{Ramo-Shockley theorem} for conductive media \cite{Rigler:2e004jh}, which relates the current $i(t)$ induced on a given electrode to the trajectory $(x(t),\dot{x}(t))$ in the phase space of a moving elementary charge $q$ and to the intrinsic electric field inside the semiconductor, the so-called (time-dependent) \textit{weighting field} $\vec{\mathbf{E}}_\mathrm{w}(t,x)$. More precisely, the induced current can be written as 
$$i(t)=-\frac{q}{V_0}\int_0^t \langle H(t-s,x(s)), \dot{x}(s)\rangle_{\R^n}\,\d s\quad\text{with}\quad H(t,x)=-\nabla\left(\frac{\de V_\mathrm{w}(t,x) \Theta(t)}{\de t}\right),$$
where $V_0>0$ is the potential applied to the electrode we are considering and all other electrodes are set to ground, $\Theta$ is the Heaviside function and $V_\mathrm{w}(t,x)$ is the so-called (time-dependent) \textit{weighting potential} such that $V_\mathrm{w}(t,x)=-\nabla \vec{\mathbf{E}}_\mathrm{w}(t,x)$. Within this formulation, the influence of the resistivity of the material on the signal formation is encoded in the weighting potential, which is determined as the solution of \textit{Maxwell’s equations in the quasi-static limit} as follows \cite{Rigler:2e004jh, Janssens:2890572}:
\begin{equation}\label{QSMdouble}
\left\{
\begin{array}{ll}
\dsy \e\,\Delta V_\mathrm{w}(t,x)= -\rho(t,x),& \\
\dsy \de_t \rho(t,x)=\div( \sigma(x)\,\nabla V_\mathrm{w}(t,x)),&
\end{array}\right.
\end{equation}
where $\rho(t,x)$ is the charge distribution, $\e=\e_0 \e_r$ is the dielectric constant in the material (which is the same for diamond and graphite) and $\sigma(x)$ is the electrical conductivity of the material. 
The function $\sigma$ is, in general, a bounded, non‑negative, discontinuous function determined by the diamond-graphite geometry.

In this paper, we study \eqref{QSMdouble} treating a unique third-order time-dependent PDE as
\begin{equation}\label{QSM1eq}
\e\de_t\Delta V_\mathrm{w}(t,x) + \div (\sigma(x) \nabla V_\mathrm{w}(t,x)) = 0\quad\text{in $[0,T]\times \Omega$},
\end{equation}
combined with a suitable initial condition (IC) and boundary conditions (BC) of mixed Neumann-Dirichlet type, where $\Omega\subseteq\R^n$ is the whole volume of the detector, which can be assumed to be an open bounded connected set. Moreover, to ensure that both the trace operator $u\mapsto u|_{\de\Omega}$ and the generalized normal trace operator $u\mapsto \de_\nu u|_{\de\Omega}$ are well defined for the data involved in the boundary conditions, we assume that $\Omega$ has the \textit{uniform $C^d$-regularity property} for some $d\geq 2$ (see Lemma \ref{extensiontrace}).

Let $C_+, C_-^i\subseteq \overline{\Omega}$, for $i=1,\ldots,m_e$, be the compact volumes of the electrodes inside $\Omega$, respectively, polarized and grounded. We assume that $\de\Omega=\overline{B_D}\cup B_N$, where 
$$B_D:=\de\Omega\cap \de\Big(C_+\cup \bigcup_{i=1}^{m_e} C_-^i\Big)\quad\text{and}\quad B_N:=\de\Omega\setminus \overline{B_D}$$ 
are non-empty disjoint open regular subsets of $\de\Omega$, called the \textit{Dirichlet boundary} and the \textit{Neumann boundary}, respectively.
Let us denote by $\nu$ the outward normal to the boundary $\de\Omega$.
We consider the following conditions:
\begin{equation}\label{IBCintro}
\left\{
\begin{array}{ll}
V_\mathrm{w}(t=0,x) = \tilde g(x),&\quad\text{in $\overline{\Omega}$},\\
V_\mathrm{w}(t,x) = g(x),&\quad\text{in $[0,T]\,\times B_D$},\\
\de_\nu V_\mathrm{w}(t,x) = 0,&\quad\text{in $[0,T]\,\times B_N$},
\end{array}\right.
\end{equation}
where $\tilde g\in H^2(\Omega)$ is the given initial data satisfying $\de_\nu \tilde g=0$ in $B_N$, and $g\in H^{\frac{3}{2}}(\de\Omega)$ is the trace of $\tilde g$ in $\de\Omega$ (see Lemma \ref{extensiontrace}).

\medskip
In this paper, we study the quasi-static Maxwell BVP from a purely mathematical perspective, establishing the existence and uniqueness of solutions to problem \eqref{QSM1eq} with initial and boundary conditions \eqref{IBCintro}. Our approach is based on the decomposition 
$$V_\mathrm{w}(t,x)=\omega(x)+u(t,x)$$
which separates the weighting potential into a static and a dynamic contribution.

The static component $\omega(x)$ solves a uniformly elliptic problem with mixed Neumann-Dirichlet boundary conditions, encoding the condition $\omega|_{B_D}\equiv g$ prescribed by the Ramo-Shockley theorem. The time-dependent component $u(t,x)$ solves a parabolic Cauchy BVP with homogeneous Neumann-Dirichlet boundary conditions and initial datum $u(0,x)=\tilde g(x)-\omega(x)$, describing the evolution of the weighting potential $V_\mathrm{w}(t,x)$ starting from the static configuration.

The existence and uniqueness of $\omega$ follows from standard elliptic techniques. The analysis of $u$ is more delicate: since the coefficient $\sigma$ may vanish in subsets of $\Omega$ with positive measure, the equation \eqref{QSM1eq} is degenerate. To overcome this difficulty, we introduce the functional variable $w(t,\cdot):=-\Delta u(t,\cdot) \in X^*$, where $X$ is a suitable Hilbert space. Using the invertibility of the homogeneous Neumann–Dirichlet Laplacian (with respect to the decomposition of the boundary $\de\Omega=\overline{B_D}\cup B_N$, see Appendix \ref{sec:A1}), we show that $w$ satisfies a non‑homogeneous abstract Cauchy problem with initial data depending on $\omega,\tilde g$. This problem is solved by means of the theory of Hille–Yosida operators and monotone operator methods (see \cite{brezis}, \cite{sinestrari}). Once $w$ is constructed, the dynamic component is recovered as 
$$u(t,\cdot)=(-\Delta_{ND})^{-1}w(t,\cdot)$$
with the regularity described in \cite{sinestrari}. In particular, we obtain the existence and uniqueness of 
$$V_\mathrm{w}\in C^1([0,T];H^1(\Omega))$$
with continuity estimates depending only on the initial data $\tilde g$ and on the structural data $\sigma,\e$.

We then investigate higher regularity properties of the solution. For the static component $\omega$, we prove local H\"older continuity in $\Omega$ by means of De Giorgi--Nash--Moser type estimates for elliptic equations, and smoothness in every open domain $\Omega'\subseteq\Omega$ where the conductivity $\sigma$ is constant, by the hypoellipticity of the Laplacian.

For the dynamic component $u$, we show that the regularity of the initial datum propagates in time within each domain where $\sigma$ is constant. More precisely, if $\sigma\equiv 0$ in an open set $\Omega_0\subseteq\Omega$, then $u'\in C^0([0,T];C^\infty(\Omega_0))$, and $u$ inherits local Sobolev and H\"older regularity from the initial datum. If instead $\sigma\equiv \sigma_1>0$ in an open set $\Omega_1\subseteq\Omega$, analogous local Sobolev and H\"older estimates hold for $u$, with bounds depending explicitly on $\sigma_1/\varepsilon$. These results are obtained by combining interior elliptic estimates in the constant-conductivity regions with the regularity theory for abstract evolution equations generated by Hille--Yosida operators.

\medskip

The paper is organized as follows.  

In Section \ref{sec:PreRes} we recall the notion of trace operators for Sobolev spaces, the definition and properties of monotone operators, and the standard notation for (local) Sobolev spaces involving time that will be needed to formulate, in the same section, the results concerning the Hille–Yosida operators of \cite{sinestrari}, including the notions of integral, $L^{p}$, and $C^{0}$ solutions, together with the corresponding existence and uniqueness theorems.

In Section \ref{sec:QSMeq}, we describe in detail the problems associated with the static component $\omega$ and the dynamic component $u$, and we state the main existence and uniqueness result for $V$, with related continuity estimate. 


Section \ref{regularity} contains the regularity results for both $u$ and $\omega$.

Finally, Section \ref{realcase} is devoted to the description of a realistic case study involving a 3D diamond detector configuration analyzed in \cite{Anderlini_2026}.

In Appendix \ref{sec:A1}, we recall the existence and uniqueness theory for the homogeneous mixed Neumann–Dirichlet problem associated with uniformly elliptic divergence–form operators.

\bigskip

\section{Preliminaries}\label{sec:PreRes}

\subsection{Sobolev spaces}
Let $1\leq p\leq \infty, m\in\N, 0<s<1$ and let $p'\in\,[1,\infty]$ be the H\"older conjugate of $p$, i.e.\,$1/p+1/p'=1$.
We use standard notation for Lebesgue spaces $L^p(\Omega)$, for Sobolev spaces $W^{m,p}(\Omega), W^{m,p}_0(\Omega)$ if $1\leq p\leq \infty$, and for $W^{-m,p}(\Omega)$ if $1< p<\infty$ (we endow the last space with the operator norm). We denote by $W^{s,p}(\de\Omega)$, with $1<p<\infty$, the standard fractional Sobolev space, and by $W^{s,p'}(\de\Omega)$ its dual space, endowed with the operator norm. 
%
%
We set
$$H^{\pm m}(\Omega):=W^{\pm m,2}(\Omega),\,\,\,H^m_0(\Omega):=W^{m,2}_0(\Omega),\,\,\,H^{\pm s}(\de\Omega):=W^{\pm s,2}(\de\Omega)$$ 
endowed with the natural inner products. We also denote $\|\cdot\|_{m}:=\|\cdot\|_{H^m(\Omega)}$ and by $(\cdot,\cdot)_{m}$ the inner product of $H^{m}(\Omega)$, for every $m\in\Z$.

In order to use the trace operators, we need to recall the following property that we require on the boundary $\de\Omega$.

\begin{defi}[Uniform $C^{d}$–regularity property, see Chapter $5$ in \cite{adams}]\label{uniformCdproperty}
Let $\Omega\subset\mathbb{R}^{n}$ be an open set and $d\in\N,d\geq 1$. We say that $\partial\Omega$ has the uniform $C^{d}$–regularity property if there exist a locally finite open covering $\{U_{j}\}_{j\in\mathbb{N}}$ of $\partial\Omega$ and a family $\{\Phi_j\}_{j\in\N}$ of injective maps $\Phi_{j} : U_{j} \to B_1(0)\subseteq\R^n$ of class $C^d$ such that:
\begin{enumerate}[(i)]
\item there exists $\delta>0$ such that
$\bigcup_{j=1}^{\infty} \Phi^{-1}_{j}\big( \{ y\in\mathbb{R}^{n} : |y|<\tfrac12 \} \big)\supset \Omega_{\delta},$ 
where $\Omega_{\delta} := \{ x\in\Omega : \operatorname{dist}(x,\partial\Omega)<\delta \}$;
\item there exists $R\in\mathbb{N}$ such that every collection of $R+1$ sets among the $U_{j}$ has empty intersection;
\item for each $j$, $\Phi_{j}(U_{j}\cap\Omega) = \{ y\in B : y_{n}>0 \};$
\item writing $\Phi_{j}=(\phi_{j,1},\dots,\phi_{j,n})$ and 
$\Phi^{-1}_{j}=(\psi_{j,1},\dots,\psi_{j,n})$, there exists $M>0$ such that
for all multi–indices $\alpha$ with $|\alpha|\le d$, for all $i=1,\dots,n$ and all $j$,
$$|D^{\alpha}\phi_{j,i}(x)| \le M, \quad |D^{\alpha}\psi_{j,i}(y)| \le M, \quad \text{for all } x\in U_{j},\,y\in B_1(0).$$
\end{enumerate}
\end{defi}

We recall the definition of the trace operator and the generalized normal trace operator.
\begin{lemma}[see Theorem $7.53$ in \cite{adams}]\label{extensiontrace}
Let $1<p<\infty$ and let $\Omega$ satisfy the condition of the uniform $C^{d}$-regularity property with $d\geq 2$. If $u\in C_0^\infty(\R^n)$, let us denote the mapping
$$u\mapsto \gamma u:=(\gamma_0 u, \gamma_1 u):=(u|_{\de\Omega}, \de_\nu u|_{\de\Omega}).$$
Then, $\gamma$ extends by continuity to the following surjective bounded linear map 
$$\gamma: 
W^{2,p}(\Omega) \longrightarrow W^{2-\frac{1}{p},p}(\de\Omega)\times W^{1-\frac{1}{p},p}(\de\Omega).$$
The operator $\gamma_0$ is called trace operator and $\gamma_1$ is called generalized normal trace operator. Moreover, $\gamma_0$ extends continuously also as a surjective bounded linear map
$$\gamma_0: W^{1,p}(\Omega)\longrightarrow W^{1-\frac{1}{p},p}(\de\Omega).$$
\end{lemma}

In the sequel, unless otherwise specified, we will consider $\Omega\subseteq\R^n$ as an open bounded set that satisfies the condition of the uniform $C^d$-regularity property with $d\geq 2$ in order to make sense of the trace operator $\gamma=(\gamma_0,\gamma_1)$.

\bigskip
If $(X,\|\cdot\|)$ is a Banach space and $T>0$, we use the standard notation for Sobolev and $C^m$ spaces involving time, denoted by $W^{m,p}(0,T;X), C^m([0,T];X)$, respectively. 

Moreover, let us consider the following local spaces. Let $m,k\in\N, \gamma\in [0,1[$. Consider an exhaustion by compact sets $\{\Omega_j\}_{j\in\N}\subseteq \Omega$ of $\Omega$ such that $\Omega_j\Subset \Omega_{j+1}\Subset\Omega$ for every $j\in\N$ and $\bigcup_{j\in\N}\Omega_j=\Omega$. We define the following Fréchet spaces:

\begin{itemize}
\item $C^k([0,T]; H^m_{loc}(\Omega))$ as the space of all functions $u:[0,T]\to H^m_{loc}(\Omega)$ such that for every compact set $K\subset \Omega$, we have $u\in C^k([0,T];H^m(K))$, endowed with the family of seminorms 
$$|u|_j:=\|u\|_{C^k([0,T];H^m(\Omega_j))}=\sum_{\ell=0}^k\sup_{t\in [0,T]}\|\de^\ell_t u(t,\cdot)\|_{H^m(\Omega_j)}\quad \forall\,j\in\N.$$
\item $C^k([0,T]; C^{m,\gamma}_{loc}(\Omega))$ as the space of all functions $u:[0,T]\to C^{m,\gamma}_{loc}(\Omega)$ such that for every compact set $K\subset \Omega$, we have $u\in C^k([0,T];C^{m,\gamma}(K))$, endowed with the family of seminorms 
$$|u|_j:=\|u\|_{C^k([0,T];C^{m,\gamma}(\Omega_j))}=\sum_{\ell=0}^k\sup_{t\in [0,T]}\|\de^\ell_t u(t,\cdot)\|_{C^{m,\gamma}(\Omega_j)}\quad \forall\,j\in\N.$$
\item $C^k([0,T]; C^{\infty}(\Omega))$ as the space of all functions $u:[0,T]\to C^{\infty}(\Omega)$ such that for every compact set $K\subset \Omega$ and $m\in\N$, we have $u\in C^k([0,T];C^{m}(K))$, endowed with the family of seminorms 
$$|u|_j:=\|u\|_{C^k([0,T];C^{j}(\Omega_j))}=\sum_{\ell=0}^k\sup_{t\in [0,T]}\|\de^\ell_t u(t,\cdot)\|_{C^{j}(\Omega_j)}\quad \forall\,j\in\N.$$
\end{itemize}

\medskip
\subsection{Monotone operators}\label{monotoneOp}

We recall the following definition of (maximal) monotone operators.

\begin{defi}[see \cite{brezis}]
Let $(H,(\cdot,\cdot)_H)$ be a Hilbert space. A linear operator $A:D(A)\subseteq H\to H$ is said to be monotone if it satisfies
$$(Au,u)_H\geq 0,\quad\text{for every $u\in D(A)$}.$$
It is called maximal monotone if, in addition, $R(Id+A)=H$, i.e.\,for every $f\in H$, there exists $u_f\in D(A)$ such that $u_f+A u_f=f.$
\end{defi}

\begin{prop}[see Proposition $7.1$ in \cite{brezis}]\label{propbrezis}
Let $A$ be a maximal monotone operator. Then
\begin{enumerate}
\item $D(A)$ is dense in $H$;
\item $A$ is a closed operator;
\item for every $\lambda>0$, $(Id+\lambda A)$ is a bijective map from $D(A)$ to $H$, $(Id+\lambda A)^{-1}$ is a bounded operator, and $\|(Id+\lambda A)^{-1}\|_{\mathcal{L}(H)}\leq 1$.
\end{enumerate}
\end{prop}

We also recall the following result to prove the surjectivity of an operator.

\begin{teo}[Minty-Browder, see \cite{browder}]\label{teo.browder}
Let $(X,\|\cdot\|_X)$ be a separable reflexive Banach space and let $T:X\to X^\ast$ be a map that is
\begin{enumerate}
\item continuous as a map between topological spaces;
\item monotone, i.e.\,$\langle T(u)-T(v)|u-v\rangle_{X^*, X}\geq 0$, for every $u,v\in X$;
\item coercive, i.e.\,$\dsy \frac{\langle T(u)|u\rangle_{X^*,X}}{\|u\|_X}\to+\infty$ as $\|u\|_X\to +\infty$.
\end{enumerate}
Then, $T$ is surjective.
\end{teo}

\medskip
\subsection{Hille-Yosida operators and abstract Cauchy problems}\label{HYop}

In this section, we introduce the definition of solutions for non-homogeneous Cauchy problems following \cite{sinestrari}.
From now on, we consider $1\leq p < \infty$. We want to give some results about the existence and uniqueness of solutions of the following \textit{non-homogeneous Cauchy problem}:
\begin{equation}\label{CauchyPbtime}
\left\{
\begin{array}{ll}
u'(t) = B\,u(t)+f(t),&\quad t\in\,[0,T], \\
u(0)=u_0,&
\end{array}\right.
\end{equation}
considering a particular class of operators $B$ as follows.

\begin{defi}[Hille-Yosida operator]\label{defi.HYoperatorsdefi}
Let $(X,\|\cdot\|)$ be a Banach space. A linear operator $B: D(B)\subseteq X\to X$ is said to be a \textit{Hille-Yosida operator} if there exist $M, \lambda_0\in \R$ such that
\begin{itemize}
\item $(\lambda - B)^{-1}$ is a linear bounded operator on $X$, for every $\lambda>\lambda_0$;
\item $\|(\lambda-\lambda_0)^m\,(\lambda-B)^{-m}\|\leq M$, for every $m\in\N$.
\end{itemize}
\end{defi}

Following \cite{sinestrari}, we give the following definitions of solution of \eqref{CauchyPbtime}.

\begin{defi}[Integral solution]\label{defi.integralsolution}
Let $f\in L^1(0,T; X)$ and $u_0\in X$. We say that $u\in C^0([0,T]; X)$ is a integral solution of \eqref{CauchyPbtime} if $\int_0^t\,u(s)\,\d s\in D(B)$ for every $t\in\,[0,T]$ and
$$u(t)=u_0+B\,\int_0^t\,u(s)\,\d s+\int_0^t\,f(s)\,\d s,\quad\text{for every $t\in\,[0,T]$}.$$
In particular, we have $u(0)=u_0$.
\end{defi}

\begin{defi}[Strict $L^p$-solution]\label{defi.Lpsolution}
Let $f :[0,T]\to X$ and $u_0\in X$. We say that $u\in W^{1,p}(0,T; X)\cap\, L^p(0,T;D(B))$ is a strict $L^p$-solution of \eqref{CauchyPbtime} if \eqref{CauchyPbtime} holds for a.e.\,$t\in\,]0,T[$ and $u(0)=u_0$. If such a solution exists, then $f\in L^p(0,T;X)$.
\end{defi}

\begin{defi}[Strict $C^0$-solution]\label{defi.C0solution}
Let $f :[0,T]\to X$ and $u_0\in X$. We say that $u\in C^1([0,T]; X)\cap\, C^0([0,T];D(B))$ is a strict $C^0$-solution of \eqref{CauchyPbtime} if \eqref{CauchyPbtime} holds for every $t\in\,[0,T]$ and $u(0)=u_0$. If such a solution exists, then $f\in C^0([0,T];X)$ and $u_0\in D(B)$.
\end{defi}
\begin{oss}
By Theorem $2.6$ in \cite{sinestrari}, if $u$ is a strict $C^0$-solution, then $u$ is an integral solution and strict $L^p$-solution, for every $1\leq p<\infty$.
\end{oss}

We are ready to state the following results about the existence and uniqueness of solutions for the problem \eqref{CauchyPbtime}.

\begin{teo}[see Theorem $5.1$ in \cite{sinestrari}]\label{teoLpsolution}
Let $(X,\|\cdot\|)$ be a Banach space and let $B:D(B)\subseteq X\to X$ be a Hille-Yosida operator. Given $u_0\in \overline{D(B)}$ and $f\in L^p(0,T;X)$, there exists a unique strict $L^p$ solution $u$ of \eqref{CauchyPbtime} such that
\begin{gather}\label{stimeLp}
\begin{split}
\|u(t)\|&\leq M\,e^{\lambda_0\,t}\,\left(\|u_0\|+\int_0^t\,e^{-\lambda_0\,s}\,\|f(s)\|\,\d s\right),\quad \text{for every $t\in\,[0,T]$},
\end{split}
\end{gather}
where $M,\lambda_0\in\R$ are as in Definition \ref{defi.HYoperatorsdefi}.
\end{teo}
\medskip

\begin{teo}[see Theorem $4.1$ in \cite{sinestrari}]\label{teo.exuniquesolHYthnonhom}
Let $(X,\|\cdot\|)$ be a Banach space and let $B:D(B)\subseteq X\to X$ be a Hille-Yosida operator. Given $u_0\in D(B)$ and $f\in W^{1,1}(0,T;X)$ such that 
$$u_1:=Bu_0+f(0)\in \overline{D(B)},$$
there exists a unique strict $C^0$ $($and so integral and $L^p$ for every $1\leq p<\infty)$ solution $u$ of \eqref{CauchyPbtime} such that
\begin{gather}\label{eq.estimateCauchytime}
\begin{split}
\|u(t)\|&\leq M\,e^{\lambda_0\,t}\,\left(\|u_0\|+\int_0^t\,e^{-\lambda_0\,s}\,\|f(s)\|\,\d s\right),\quad \text{for every $t\in\,[0,T]$},\\
\|u'(t)\|&\leq M\,e^{\lambda_0\,t}\,\left(\|u_1\|+\int_0^t\,e^{-\lambda_0\,s}\,\|f'(s)\|\,\d s\right),\quad \text{for every $t\in\,[0,T]$},
\end{split}
\end{gather}
where $M,\lambda_0\in\R$ are as in Definition \ref{defi.HYoperatorsdefi}.
\end{teo}

\bigskip

\section{Weak solutions to Quasi-static Maxwell BVP}\label{sec:QSMeq}

From now on, for notational convenience, we drop the dependence of the weighting potential $V_\mathrm{w}$ on $\mathrm{w}$.
We consider the following \textit{Quasi-static Maxwell boundary value problem with initial data}:
\begin{equation}\label{eq.BVPIC}
\left\{
\begin{array}{ll}
\e \de_t\Delta V(t,x) + \div (\sigma(x) \nabla V(t,x)) = 0,&\quad\text{in $[0,T]\,\times \Omega$}, \\
V(t=0,x) = \tilde g(x),&\quad\text{in $\overline{\Omega}$},\\
V(t,x) = g(x),&\quad\text{in $[0,T]\,\times B_D$},\\
\de_\nu V(t,x) = 0,&\quad\text{in $[0,T]\,\times B_N$},
\end{array}\right.
\end{equation}
where we consider the following\medskip

\begin{hypothesis}\label{Hgsigmae}
Let $\Omega\subseteq\R^n$ be a connected bounded open set satisfying the uniform $C^{d}$-regularity property with $d\ge2$. We assume the following:
\begin{itemize}
\item $\tilde g\in H^2(\Omega)$ is the initial condition of $V$, and $g\in H^{\frac{3}{2}}(\de\Omega)$ is the Dirichlet-boundary datum prescribed by the Ramo-Shockley theorem (see \cite{Anderlini_2026} for more details). Moreover, since $\Omega$ satisfies the uniform $C^{d}$-regularity property with $d\ge2$, the generalized normal trace operator $\gamma_1$ is well defined (see Lemma \ref{extensiontrace}), and we require the compatibility condition
$$\de_\nu\tilde g|_{B_N}=\gamma_1(\tilde g) |_{B_N}=0;$$
%
\item $\sigma\in L^\infty(\Omega)$ is a non-negative function, denoting the electrical conductivity of the material, which is independent of time;
\item $\e=\e_r\e_0>0$ is a positive constant, denoting the dielectric constant of the material.
\end{itemize}
\end{hypothesis}

\medskip

Defining $a(x):=\frac{\sigma(x)}{\e}$, we can rewrite the equation in \eqref{eq.BVPIC} as
$$\de_t \Delta V(t,x) + \div (a(x) \nabla V(t,x)) = 0,\quad\text{in $[0,T]\,\times \Omega$}.$$

%

The strategy to prove the existence of a solution of \eqref{eq.BVPIC} is to divide the problem in two sub-problems as follows: we will prove that \eqref{eq.BVPIC} admits a unique solution decomposed as
\begin{equation}\label{eq.solutionsplit}
V(t,x) = \omega(x) + u(t,x)
\end{equation}
where $\omega$ is the unique (static) weak solution (in the sense of Proposition \ref{solPBstatic}) of the \textit{elliptic BVP with Neumann-Dirichlet BC}
\begin{equation}\label{eq.BVPomega}
\left\{
\begin{array}{ll}
\dsy -\div \big( A(x) \nabla \omega(x) \big) = 0,&\quad\text{in $\Omega$}, \\
\omega(x) = g(x),&\quad\text{in $B_D$},\\
\de_\nu \omega(x) = 0,&\quad\text{in $B_N$},
\end{array}\right.
\end{equation}
where $A(x):=a(x)+1$, and $u$ is the unique (time-dependent) weak solution (in the sense of Proposition \ref{solPBtime}) of the \textit{parabolic degenerate Cauchy BVP}
\begin{equation}\label{eq.BVPICv}
\left\{
\begin{array}{ll}
\de_t \Delta u(t,x) + \div (a(x) \nabla u(t,x)) = \Delta\omega,&\quad\text{in $[0,T]\,\times \Omega$}, \\
u(t=0,x) = -\tilde \omega(x),&\quad\text{in $\overline{\Omega}$},\\
u(t,x) = 0,&\quad\text{in $[0,T]\,\times B_D$},\\
\de_\nu u(t,x) = 0,&\quad\text{in $[0,T]\,\times B_N$},
\end{array}\right.
\end{equation}
where $\tilde \omega := \omega- \tilde g$. (We stress that the initial condition and boundary conditions in \eqref{eq.BVPICv} are coherent thanks to the hypothesis on $\tilde g$).

\medskip

To state the main result, let us define
$$X:=\left\{u\in H^1(\Omega)\,\,|\,\,R_{B_D}u=0\right\}$$
as the space of functions in $H^1(\Omega)$ with zero-trace on $B_D$ endowed with the norm $\|\nabla\cdot\|_{0}$ (see Section \ref{sec:PB1} for more details), and let $(-\Delta_{ND})^{-1}$ be the solution operator of mixed Neumann-Dirichlet BVP associated to the Laplacian as in Appendix \ref{sec:A1}.

\begin{defi}\label{solutionBVP1}
Assume Hypothesis \ref{Hgsigmae}.
We say that $V\in C^1([0,T];H^1(\Omega))$ is a weak solution of problem \eqref{eq.BVPIC} if there exist $\omega\in H^1(\Omega), u\in C^1([0,T]; H^1(\Omega))$ such that 
$$V(t,x)=\omega(x)+u(t,x)\quad\text{for every $t\in [0,T]$, for a.e.\,$x\in\Omega$,} $$
such that
\begin{itemize}
\item $\omega-\tilde g\in X$ and $\omega$ satisfies
\begin{equation*}
\int_\Omega\,\left(\frac{\sigma(x)}{\e}+1\right) \langle \nabla \omega(x),\nabla\varphi(x) \rangle\,dx=0,\quad\text{for every  $\varphi\in X$}.
\end{equation*}
\item $u(t,\cdot)=(-\Delta_{ND})^{-1} w(t,\cdot)$ for every $t\in [0,T]$ such that $w\in C^1([0,T];X^*)$ is the unique solution of
$$\langle w'(t,\cdot)|\varphi\rangle_{X^*,X} +\int_\Omega\,\frac{\sigma(x)}{\e}\langle\nabla (-\Delta_{ND})^{-1} w(t,\cdot),\nabla\varphi\rangle\,dx=\int_\Omega\langle \nabla\omega,\nabla\varphi\rangle\,dx$$
for every $\varphi\in X$, $t\in[0,T]$, with the initial condition $w(0,\cdot) = \Delta(\omega-\tilde g)$.
\end{itemize}
\end{defi}


\medskip
Our main result of this section is the following

\begin{teo}\label{Solution1}
Assume Hypothesis \ref{Hgsigmae}. Then, there exists a unique weak solution $V\in C^1([0,T];H^1(\Omega))$ of \eqref{eq.BVPIC} in the sense of Definition \ref{solutionBVP1}.
Moreover, we have the estimate
\begin{align*}
\|V\|_{C^1([0,T];H^1(\Omega))}&\leq \|\omega\|_{H^{1}(\Omega)}+c_\Omega\|u\|_{C^1([0,T];X)}\\
&\leq 2\max\{1,c_{\Omega}^2\}\max\{1,T\}\left(1+\|\sigma\|_{L^\infty(\Omega)}/\e\right)^2\|\tilde g\|_{H^1(\Omega)},
\end{align*}
where $c_\Omega>0$ is the constant of equivalence between the $H^1(\Omega)$-norm and $\|\nabla\cdot\|_{L^2(\Omega)}$ in $X$ as in Proposition \ref{prop.equivnormX}.
\end{teo}

\bigskip
\subsection{Analysis of the stationary part}\label{sec:PB1}

We want to prove the existence of $\omega$ as the unique weak solution of \eqref{eq.BVPomega}. Let us start with some consideration on the suitable space we will use. Since $B_D$ is a submanifold of $\de\Omega$, we have that $(\cdot|_{B_D}) : H^{\frac{1}{2}}(\de\Omega)\hookrightarrow H^{\frac{1}{2}}(B_D)$ is a continuous map, so the trace operator on $B_D$
$$R_{B_D}:=(\cdot|_{B_D})\circ \gamma_0 :H^{1}(\Omega)\to H^{\frac{1}{2}}(B_D)$$
is a bounded linear operator. 
Let us define the space
\begin{equation}\label{defX}
X:=\left\{u\in H^1(\Omega)\,\,|\,\,R_{B_D}u=0\right\}
\end{equation}
and we observe that, by the continuity of $R_{B_D}$, $X$ is a Hilbert space endowed with the inner product $(\cdot,\cdot)_1$.
(We stress that $R_{B_D}$ is well posed on $X$ by definition.)

\begin{prop}\label{prop.equivnormX}
$\|\cdot\|_{H_0^1(\Omega)}:=\|\nabla\cdot\|_{0}$ is an equivalent norm on $X$, i.e.\,there exists a constant $c_\Omega>0$ such that
$$\frac{1}{c_\Omega}\|\varphi\|_1\leq \|\varphi\|_{H_0^1(\Omega)}\leq c_\Omega\|\varphi\|_1,\quad\text{for every $\varphi\in X$}.$$
\end{prop}
\begin{proof}
Trivially $\|\varphi\|_{H_0^1(\Omega)}\leq\|\varphi\|_{1}$, for every $\varphi\in X$. To prove the other inequality, we argue by contradiction. 
Suppose that for every $c_k:=1/k>0$, there exist $\varphi_k\in X$ such that $\|\varphi_k\|_0=1$ and
$\|\varphi_k\|_{H_0^1(\Omega)}\leq c_k \|\varphi_k\|_{1}$. Thus,
$$\|\nabla\varphi_k\|^2_{0}=\|\varphi_k\|^2_{H_0^1(\Omega)}\leq \frac{1}{k^2}\|\varphi_k\|_1^2= \frac{1}{k^2}\|\varphi_k\|^2_0+\frac{1}{k^2}\|\nabla\varphi_k\|^2_{0},$$
so that
\begin{equation}\label{eq.stimanormeeq}
\|\nabla\varphi_k\|^2_{0}\leq \frac{1}{k^2-1}.
\end{equation}
Since $\|\varphi_k\|_0=1$, then $\|\varphi_k\|_{1}\leq 2$, showing that $(\varphi_k)_{k\in\N}$ is a bounded sequence in $X$ with respect to $\|\cdot\|_1$. Hence, by \textit{Banach-Alaoglu theorem} and \textit{Rellich-Kondrachov embedding}, there exists $\varphi\in H^1(\Omega)$ such that $\varphi_k\rightharpoonup \varphi$ in $H^1(\Omega)$ and $\varphi_k\to \varphi$ in $L^2(\Omega)$.
By \eqref{eq.stimanormeeq}, we have $\nabla\varphi_k\to 0$ in $L^2(\Omega)$ as $k\to+\infty$, and $\nabla\varphi=0$, showing $\varphi_k\to\varphi$ in $X$. Since $R_{B_D}: H^1(\Omega)\to H^{\frac{1}{2}}(B_D)$ is a bounded operator, then $0=R_{B_D}\varphi_k\to R_{B_D}\varphi$ in $H^{\frac{1}{2}}(B_D)$, so that $R_{B_D}\varphi=0$, implying $\varphi\in X$.

Now, by \textit{Poincaré–Wirtinger inequality} (recall that $\Omega$ is a connected bounded open set), we have
that $\varphi$ is constant on $\overline{\Omega}$. On the other hand, $\varphi|_{B_D}=R_{B_D}\varphi=0$, implying $\varphi=0$ in $\overline{\Omega}$. But, by construction, $\|\varphi_k\|_{0}=1$ for every $k\in\N$, so $\|\varphi\|_{0}=1$, which is a contradiction.
\end{proof}\medskip

From now on, we consider $X$ endowed with the norm $\|\cdot\|_X:=\|\nabla\cdot\|_0$. We consider the continuous dual space $X^*$ of $X$ endowed with the operator norm 
$$\|F\|_{X^*}:=\sup\big\{|\langle F|\varphi\rangle_{X^*,X}|\,\,:\,\,\varphi\in X,\,\|\varphi\|_X\leq 1\big\},$$
where $\langle\cdot|\cdot\rangle_{X^*,X}$ is the duality between $X^*$ and $X$.

\medskip

We are in position to state the definition of weak solution to problem \eqref{eq.BVPomega}.

\begin{defi}\label{weakomega}
Assume Hypothesis \ref{Hgsigmae}. We say that $\omega\in H^1(\Omega)$ is a weak solution of \eqref{eq.BVPomega} if $\omega-\tilde g\in X$ and 
\begin{equation}\label{equazione1}
\int_\Omega\,\left(\frac{\sigma(x)}{\e}+1\right) \langle \nabla \omega(x),\nabla\varphi(x) \rangle\,dx=0,\quad\text{for every  $\varphi\in X$}.
\end{equation}
\end{defi}

We can now prove the following existence and uniqueness result for the stationary component of $V$. The proof relies on a standard Lax--Milgram argument for elliptic problems with mixed boundary conditions, recalled in Appendix \ref{sec:A1}.

\begin{prop}\label{solPBstatic}
Assume Hypothesis \ref{Hgsigmae}.
Then, there exists a unique weak solution $\omega\in H^1(\Omega)$ of \eqref{eq.BVPomega} as in Definition \ref{weakomega}.
Moreover, we have
\begin{equation}\label{omegaH1}
\|\omega\|_{H^1(\Omega)}\leq 2\max\{1,c_\Omega\}\,\left(1+\|\sigma\|_{L^\infty(\Omega)}/\e\right)\|\tilde g\|_{H^{1}(\Omega)},
\end{equation}
\begin{equation}\label{omegatildeH1}
\|\nabla(\omega-\tilde g)\|_{L^2(\Omega)}\leq (1+\|\sigma\|_{L^\infty(\Omega)}/\e)\|\nabla\tilde g\|_{L^2(\Omega)},
\end{equation}
where $c_\Omega>0$ is the constant of equivalence between the $H^1(\Omega)$-norm and $\|\nabla\cdot\|_{L^2(\Omega)}$ in $X$ as in Proposition \ref{prop.equivnormX}.
\end{prop}
\begin{proof}
Let $\tilde \omega:=\omega-\tilde g$. Then, $\tilde \omega$ formally satisfies
\begin{equation}\label{eq.BVPomegatilde}
\left\{
\begin{array}{ll}
\dsy -\div \big( A(x) \nabla \tilde\omega ) = \div(A(x)\nabla\tilde g),&\quad\text{in $\Omega$}, \\
\tilde \omega = 0,&\quad\text{in $B_D$},\\
\de_\nu \tilde \omega = -\de_\nu \tilde g=0,&\quad\text{in $B_N$}.
\end{array}\right.
\end{equation}
By hypothesis on $\tilde g$, the functional $-\div \big( A \nabla \tilde g \big): X\to \R$ defined as
$$\langle-\div \big( A \nabla \tilde g \big)|\varphi\rangle_{X^*,X}:=( A \nabla \tilde g, \nabla\varphi )_0,\quad\text{for every $\varphi\in X$}$$
is well-posed and $-\div \big( A \nabla \tilde g \big)\in X^*$ with the estimate
\begin{equation}\label{stimadivA}
\|-\div \big( A \nabla \tilde g \big)\|_{X^*}\leq \left(\esssup_{\Omega}A\right)\,\|\nabla \tilde g\|_{0}\leq \left(1+\|a\|_{L^\infty(\Omega)}\right)\,\|\tilde g\|_{1}.
\end{equation}
Since $A$ is uniformly bounded from below $$\essinf_{\Omega}A=\essinf_{\Omega} a +1\geq 1,$$
by the results in Appendix \ref{sec:A1}, there exists a unique weak solution $\tilde\omega\in X$ of problem \eqref{eq.BVPomegatilde}, i.e.\,
\begin{equation}\label{eq.BVPND2weakk}
\dsy ( A \nabla \tilde\omega,\nabla\varphi )_0 =- (A\nabla\tilde g,\nabla\varphi)_0,\quad\text{for every  $\varphi\in X$},
\end{equation}
or, equivalently,
$$(A\nabla\omega,\nabla\varphi)_0=(A\nabla(\tilde\omega+\tilde g),\nabla\varphi)_0=0,\quad\text{for every  $\varphi\in X$}.$$
Moreover, combining estimate \eqref{eq.stimacontF} in Appendix \ref{sec:A1} and estimate \eqref{stimadivA}, we have
\begin{align*}
\|\omega\|_{1}&\leq c_\Omega\|\tilde\omega\|_{X}+\|\tilde g\|_{1}\leq c_\Omega\|-\div \big( A \nabla \tilde g \big)\|_{X^*}+\|\tilde g\|_{1}\\
&\leq 2\max\{1,c_\Omega\}\,\left(1+\|a\|_{L^\infty(\Omega)} \right)\|\tilde g\|_{1},
\end{align*}
completing the proof since \eqref{omegatildeH1} descends from estimate \eqref{stimadivA}.
\end{proof}

\begin{oss}\label{tildeomegaXtilde}
In the proof of Proposition \ref{solPBstatic}, we proved that $\tilde \omega$ satisfies \eqref{eq.BVPomegatilde} weakly, or equivalently \eqref{eq.BVPND2weakk}. In particular, by the boundary conditions in \eqref{eq.BVPomegatilde}, the functional
$$\langle -\Delta\tilde \omega|\varphi\rangle_{X^*,X}:=(\nabla\tilde \omega, \nabla\varphi)_0\quad\forall\,\varphi\in X$$
is well-posed and $-\Delta\tilde \omega\in X^*$. Setting $F:=-\Delta\tilde \omega$, then $\tilde\omega$ satisfies 
\begin{equation*}
\left\{
\begin{array}{ll}
\dsy -\Delta \tilde\omega = F,&\quad\text{in $\Omega$}, \\
\tilde \omega = 0,&\quad\text{in $B_D$},\\
\de_\nu \tilde \omega =0,&\quad\text{in $B_N$},
\end{array}\right.
\end{equation*}
weakly, or equivalently
$$(\nabla\tilde \omega, \nabla\varphi)_0=\langle F|\varphi\rangle_{X^*,X}\quad\forall\,\varphi\in X.$$
Using the same notation of Appendix \ref{sec:A1}, we have $\tilde\omega\in \tilde X$.
\end{oss}

\bigskip
\subsection{Analysis of the time-dependent part}\label{sec:PB2}

To solve \eqref{eq.BVPICv}, we reformulate the problem as an abstract Cauchy problem of the form \eqref{CauchyPbtime} in Section~\ref{HYop}, so that we may apply Theorem~\ref{teo.exuniquesolHYthnonhom}.
The key idea is to invert the Laplacian $-\Delta$ with respect to the mixed Neumann--Dirichlet boundary conditions satisfied by $u$.  
More precisely, since $u$ satisfies homogeneous Neumann--Dirichlet conditions in \eqref{eq.BVPICv}, we work with the solution operator associated with the problem introduced in Appendix \ref{sec:A1}.  
Accordingly, we regard the time-dependent component as
\[
u(t,\cdot)\in \tilde X := R\,\left((-\Delta_{ND})^{-1}\right),
\qquad t\in[0,T],
\]
and rewrite \eqref{eq.BVPICv} as an evolution equation for 
\[
w(t,\cdot):=-\Delta u(t,\cdot)\in X^*,
\]
which fits into the Hille--Yosida framework.

Let us start with some considerations. At first, we observe that the functionals
\begin{align*}
\langle-\Delta u(t,\cdot)|\varphi\rangle_{X^*,X}&:=(\nabla u(t,\cdot),\nabla\varphi)_0\quad&\forall\,\varphi\in X,\\
\langle-\div(a \nabla u(t,\cdot))|\varphi\rangle_{X^*,X}&:=(a\nabla u(t,\cdot),\nabla\varphi)_0\quad&\forall\,\varphi\in X,
\end{align*}
are well-posed and are in $X^*$. 
Moreover, considering $\omega,\tilde\omega$ as in Section \ref{sec:PB1}, we can define
\begin{align*}
\langle -\Delta\tilde \omega|\varphi\rangle_{X^*,X}&:=(\nabla\tilde \omega, \nabla\varphi)_0\quad&\forall\,\varphi\in X,\\
\langle -\Delta\omega|\varphi\rangle_{X^*,X}&:=(\nabla\omega, \nabla\varphi)_0\quad&\forall\,\varphi\in X,\\
\langle-\div(a \nabla \tilde \omega)|\varphi\rangle_{X^*,X}&:=(a\nabla \tilde \omega,\nabla\varphi)_0\quad&\forall\,\varphi\in X,
\end{align*}
which are well-posed in $X^*$, and coherent with respect to the boundary conditions on $\omega,\tilde\omega$ in \eqref{eq.BVPomega} and \eqref{eq.BVPomegatilde} (recall Remark \ref{tildeomegaXtilde} and that $a$ is a scalar function), with the estimate
\begin{align*}
\|-\Delta\tilde\omega\|_{X^*}&\leq \left(1+\|a\|_{L^\infty(\Omega)}\right)\|\nabla\tilde g\|_{0},\\
\|-\Delta\omega\|_{X^*}&\leq 2\max\{1,c_\Omega\}\left(1+\|a\|_{L^\infty(\Omega)}\right)\|\tilde g\|_{1},\\
\|-\div(a \nabla \tilde \omega)\|_{X^*}&\leq \|a\|_{L^\infty(\Omega)}\left(1+\|a\|_{L^\infty(\Omega)}\right)\|\nabla\tilde g\|_{0}.
\end{align*}
Using the initial condition in \eqref{eq.BVPICv}, we have
$$w(t=0,\cdot)=-\Delta u(t=0,\cdot)=\Delta\tilde\omega\in X^*.$$

Hence, we can rewrite \eqref{eq.BVPICv} as
\begin{equation}\label{eq.BVPICweak}
\left\{
\begin{array}{ll}
\de_t w(t,\cdot) = \div \big(a \nabla (-\Delta_{ND})^{-1} w(t,\cdot)\big) - \Delta\omega,&\quad\text{in $X^*$, $t\in [0,T]$}, \\
w(0,\cdot) = \Delta\tilde\omega\in X^*.&
\end{array}\right.
\end{equation}
We stress that the Neumann-Dirichlet BC on $u$ are encoded in the request of $u(t,\cdot):=(-\Delta_{ND})^{-1} w(t,\cdot)$ for every $t\in [0,T]$.

\medskip

\begin{prop}\label{exunw}
Assume Hypothesis \ref{Hgsigmae}. There exists a unique strict $C^0$ (integral and strict $L^p$ for every $1\leq p < \infty$) solution $w$ of \eqref{eq.BVPICweak}. In particular, $w(t,\cdot)\in C^1([0,T];X^*)$ such that $w$ satisfies the equation in \eqref{eq.BVPICweak} for every $t\in [0,T]$, and 
\begin{align*}
\|w(t,\cdot)\|_{X^*}&\leq 2\max\{1,t\}\max\{1,c_\Omega\}\left(1+\|\sigma\|_{L^\infty(\Omega)}/\e\right)\|\tilde g\|_{H^1(\Omega)}, \\
\|w'(t,\cdot)\|_{X^*}&\leq 2\max\{1,c_\Omega\}\max\{1,\|\sigma\|_{L^\infty(\Omega)}/\e\}\left(1+\|\sigma\|_{L^\infty(\Omega)}/\e\right)\|\tilde g\|_{H^1(\Omega)},
\end{align*}
for every $t\in [0,T]$, where $c_\Omega>0$ is the constant of equivalence between the $H^1(\Omega)$-norm and $\|\nabla\cdot\|_{L^2(\Omega)}$ in $X$ as in Proposition \ref{prop.equivnormX}.

In particular,
$$\|w\|_{C^1([0,T];X^*)}\leq 2\max\{1,c_\Omega\}\max\{1,T\}\left(1+\|\sigma\|_{L^\infty(\Omega)}/\e\right)^2\|\tilde g\|_{H^1(\Omega)}.$$
\end{prop}


\medskip
To prove Proposition \ref{exunw}, we need some preliminary results. With the same notation of Definition \ref{defi.HYoperatorsdefi} and Theorem \ref{teo.exuniquesolHYthnonhom}, we set
$$B:=\div \big(a \nabla (-\Delta_{ND})^{-1} \cdot\big):D(B):=X^*\to X^*$$
which is well-posed by previous computations (recall that $\tilde X\subseteq X$ with $\|\cdot\|_{\tilde X}=\|\cdot\|_X$). Moreover, for every $\zeta\in X^*, \varphi\in X$, we have
\begin{align*}
|\langle B\zeta|\varphi\rangle_{X^*,X}|&=|(a\nabla(-\Delta_{ND})^{-1}\zeta,\nabla\varphi)_0|\\
&\leq \|a\|_{L^\infty(\Omega)}\|\nabla(-\Delta_{ND})^{-1}\zeta\|_0\|\nabla\varphi\|_0\\
&= \|a\|_{L^\infty(\Omega)}\|(-\Delta_{ND})^{-1}\zeta\|_{\tilde X}\|\varphi\|_X\\
&\leq \|a\|_{L^\infty(\Omega)}\|\zeta\|_{X^*}\|\varphi\|_X,
\end{align*}
where in the last inequality we used the continuity bound for $(-\Delta_{ND})^{-1}$ as $\|(-\Delta_{ND})^{-1}\|_{\mathcal{L}(X^*,\tilde X)}\leq 1$ proved in Appendix \ref{sec:A1}. Thus, we proved $B\in \mathcal{L}(X^*)$ with the estimate
$$\|B\|_{\mathcal{L}(X^*)}\leq \|\sigma\|_{L^\infty(\Omega)}/\e.$$

\medskip

We want to prove that $B$ is a Hille-Yosida operator as in Definition \ref{defi.HYoperatorsdefi}. For this aim, we apply results of functional analysis about monotone operators as in Section \ref{monotoneOp}.



\begin{prop}\label{BHYoperator}
$B:X^*\to X^*$ is a Hille-Yosida operator as in Definition \ref{defi.HYoperatorsdefi}.
\end{prop}
\begin{proof}
At first, we prove that $(\lambda-B)^{-1}$ exists and it is a bounded operator on $X^*$, for every $\lambda>\lambda_0:=0$. For this aim, we define 
$$A:=(-\Delta_{ND})^{-1}(-B)(-\Delta|_{\tilde X})=(-\Delta_{ND})^{-1}(-\div(a\nabla\cdot)):\tilde X\to \tilde X$$
and we prove that $A$ is a maximal monotone operator on $\tilde X$, so, by Proposition \ref{propbrezis}.3, we obtain $(Id_{\tilde X}+\sigma A)^{-1}$ is a bounded operator, for every $\sigma>0$.

First of all, $A$ is a monotone operator. Recall that, if $\eta,\phi\in \tilde X$, then $-\Delta\eta\in X^*$ and $\langle -\Delta\eta|\phi\rangle_{X^*,X}=(\nabla\eta,\nabla\phi)_0=(\eta,\phi)_{\tilde X}$. Thus,
\begin{align*}
( A\phi,\phi)_{\tilde X}&=( (-\Delta_{ND})^{-1}(-\div(a\nabla\phi)),\phi)_{\tilde X}\\
&=( \nabla(-\Delta_{ND})^{-1}(-\div(a\nabla\phi)),\nabla\phi)_{0}\\
&=\langle-\Delta(-\Delta_{ND})^{-1}(-\div(a\nabla\phi))|\phi\rangle_{X^*,X}\\
&=\langle-\div(a\nabla\phi)|\phi\rangle_{X^*,X}\\
&=(a\nabla\phi,\nabla\phi)_0\geq \left(\essinf_{\Omega}\,a\right)\,\|\nabla\phi\|^2_0\geq 0.
\end{align*}

$A$ is a maximal monotone operator, i.e.\,for every $f\in \tilde X$, there exists $\phi\in \tilde X$ such that $\phi+A\phi=f$, i.e.\,the map
$$Id_{\tilde X}+A:\tilde X\to \tilde X$$
is surjective. Let $i:\tilde X\to (\tilde X)^*$ be the identification (isometric isomorphism) of Hilbert spaces with their duals, where $(\tilde X)^*$ is endowed with the operator norm. Then, we want to apply the \textit{Minty-Browder theorem} to the map
$$i\circ(Id_{\tilde X}+A):\tilde X\to (\tilde X)^*.$$
Trivially, $i(Id_{\tilde X}+A)$ is a continuous map. 
Moreover, it is monotone in the sense of Theorem \ref{teo.browder} since $A$ is monotone by previous computations. Indeed, for every $\phi\in \tilde X$,
\begin{align*}
\langle i(Id_{\tilde X}+A)\phi|\phi\rangle_{(\tilde X)^*, \tilde X}=((Id_{\tilde X}+A)\phi,\phi)_{\tilde X}=\|\phi\|^2_{\tilde X}+(A\phi,\phi)_{\tilde X}\geq 0.
\end{align*}
Eventually, $i(Id_{\tilde X}+A)$ is coercive in the sense of Theorem \ref{teo.browder} since, for every $\phi\in \tilde X$, we have
\begin{align*}
\frac{\langle i(Id_{\tilde X}+A)\phi|\phi\rangle_{(\tilde X)^*, \tilde X}}{\|\phi\|_{\tilde X}}&=\frac{\|\phi\|^2_{\tilde X}+(A\phi,\phi)_{\tilde X}}{\|\phi\|_{\tilde X}}\geq\|\phi\|_{\tilde X}\xrightarrow{\|\phi\|_{\tilde X}\to+\infty} +\infty.
\end{align*}
Thus, by Theorem \ref{teo.browder}, the map $i(Id_{\tilde X}+A):\tilde X\to (\tilde X)^*$ is surjective, implying $(Id_{\tilde X}+A):\tilde X\to \tilde X$ is surjective, and showing that $A$ is a maximal monotone operator on $\tilde X$.

By Proposition \ref{propbrezis}, for every $\sigma>0$, $(Id_{\tilde X}+\sigma A)^{-1}$ is a bounded operator such that
$$\|(Id_{\tilde X}+\sigma A)^{-1}\|_{\mathcal{L}(\tilde X)}\leq 1.$$
Thus, $Id_{X^*}+\sigma (-B)=(-\Delta)(Id_{\tilde X}+\sigma A)(-\Delta_{ND})^{-1}$ is bijective on $X^*$. Using the continuity estimates for $-\Delta$ and $(-\Delta_{ND})^{-1}$ in Appendix \ref{sec:A1}, we have
\begin{align*}
\|(Id_{X^*}+\sigma& (-B))^{-1}\|_{\mathcal{L}(X^*)}\\
&\leq\|-\Delta\|_{\mathcal{L}(\tilde X,X^*)}\|(-\Delta_{ND})^{-1}\|_{\mathcal{L}(X^*,\tilde X)}\|(Id_{\tilde X}+\sigma A)^{-1}\|_{\mathcal{L}(\tilde X)}\\
&\leq \|-\Delta\|_{\mathcal{L}(\tilde X,X^*)}\|(-\Delta_{ND})^{-1}\|_{\mathcal{L}(X^*,\tilde X)}\leq 1.
\end{align*}
Now, if we define $\lambda:=\frac{1}{\sigma}$, then 
\begin{align*}
(Id_{X^*}+\sigma (-B))^{-1}=\sigma^{-1}(\sigma^{-1} Id_{X^*}+(-B))^{-1}=\lambda (\lambda Id_{X^*}+(-B))^{-1},
\end{align*}
showing
$$\|(\lambda Id_{X^*}+ (-B))^{-1}\|_{\mathcal{L}(X^*)}\leq \frac{1}{\lambda}\quad\text{for every $\lambda>0$}.$$
Moreover, $B$ satisfies the second hypothesis of Definition \ref{defi.HYoperatorsdefi} with $\lambda_0:=0, M:=1$. Indeed, for every $m\in\N, \lambda>0$,
$$\|(\lambda-\lambda_0)^m(\lambda Id_{X^*}+ (-B))^{-m}\|_{\mathcal{L}(X^*)}\leq (\lambda-0)^m\cdot \left(\frac{1}{\lambda}\right)^m=1.$$
Thus, $B:X^*\to X^*$ is a Hille-Yosida operator in the sense of Definition \ref{defi.HYoperatorsdefi}.
\end{proof}

\medskip

\begin{proof}[Proof of Proposition \ref{exunw}]
By previous computations, the inhomogeneous term $f:=-\Delta\omega$ belongs to $X^*$, and, since $f$ does not depend on $t$, $f\in W^{1,1}(0,T;X^*)$. Moreover, $w_0=\Delta\tilde\omega\in X^*$ and, using Remark \ref{tildeomegaXtilde}, we have
\begin{align*}
w_1&=Bw_0+f(0)=\div(a\nabla(-\Delta_{ND})^{-1}\Delta\tilde\omega)-\Delta\omega=-\div(a\nabla\tilde\omega)-\Delta\omega\in X^*.
\end{align*}
Since $B$ is a Hille-Yosida operator by Proposition \ref{BHYoperator}, we can apply Theorem \ref{teo.exuniquesolHYthnonhom} to \eqref{eq.BVPICweak} in $X^*$, obtaining the existence and uniqueness of $w$ with the regularity as in the statement. Moreover, the estimates \eqref{eq.estimateCauchytime} hold as
\begin{gather}\label{eq.estimatewwprimo}
\begin{split}
\|w(t,\cdot)\|_{X^*}&\leq 1 e^0\left(\|w(0,\cdot)\|_{X^*}+\int_0^t e^0\|f\|_{X^*}\,\d s\right)\\
&=\|\Delta\tilde\omega\|_{X^*}+t\|-\Delta\omega\|_{X^*}\\
&=2\max\{1,t\}\max\{1,c_\Omega\}\left(1+\|a\|_{L^\infty(\Omega)}\right)\|\tilde g\|_{1},\\
\|w'(t,\cdot)\|_{X^*}&\leq \|w_1\|_{X^*}\\
&\leq\|-\div(a\nabla\tilde\omega))\|_{X^*}+\|-\Delta\omega\|_{X^*}\\
&\leq 2\max\{1,c_\Omega\}\max\{1,\|a\|_{L^\infty(\Omega)}\}\left(1+\|a\|_{L^\infty(\Omega)}\right)\|\tilde g\|_{1},
\end{split}
\end{gather}
for every $t\in [0,T]$, completing the proof.
\end{proof}

\medskip

We are now in position to state the following definition of weak solutions of \eqref{eq.BVPICv}.
\begin{defi}\label{weakudefi}
Assume Hypothesis \ref{Hgsigmae}. We say that $u\in C^1([0,T];\tilde X)$ is a weak solution of \eqref{eq.BVPICv} if there exists $w\in C^1([0,T];X^*)$ such that $u(t,\cdot)=(-\Delta_{ND})^{-1} w(t,\cdot)$ for every $t\in [0,T]$ and $w$ is the unique solution of
$$\langle w'(t,\cdot)|\varphi\rangle_{X^*,X} +\int_\Omega\,\frac{\sigma(x)}{\e}\langle\nabla (-\Delta_{ND})^{-1} w(t,\cdot),\nabla\varphi\rangle\,dx=\int_\Omega\langle \nabla\omega,\nabla\varphi\rangle\,dx$$
for every $\varphi\in X$, $t\in[0,T]$, with the initial condition $w(0,\cdot) = \Delta\tilde\omega$, where $\omega$ is the unique weak solution of the problem \eqref{eq.BVPomega}.
\end{defi}

We can prove the following result about the existence and uniqueness of the dynamic part of the weighting potential.

\begin{prop}\label{solPBtime}
Assume Hypothesis \ref{Hgsigmae}. Then, there exists a unique weak solution $u\in C^1([0,T];\tilde X)$ of \eqref{eq.BVPICv} in the sense of Definition \ref{weakudefi}. Moreover, we have the estimates
\begin{align*}
\|u(t,\cdot)\|_{X}&\leq 2\max\{1,t\}\max\{1,c_\Omega\}\left(1+\|\sigma\|_{L^\infty(\Omega)}/\e\right)\|\tilde g\|_{H^1(\Omega)}, \\
\|u'(t,\cdot)\|_{X}&\leq 2\max\{1,c_\Omega\}\max\{1,\|\sigma\|_{L^\infty(\Omega)}/\e\}\left(1+\|\sigma\|_{L^\infty(\Omega)}/\e\right)\|\tilde g\|_{H^1(\Omega)},
\end{align*}
for every $t\in [0,T]$. In particular,
$$\|u\|_{C^1([0,T];\tilde X)}\leq 2\max\{1,c_\Omega\}\max\{1,T\}\left(1+\|\sigma\|_{L^\infty(\Omega)}/\e\right)^2\|\tilde g\|_{H^1(\Omega)}.$$
\end{prop}


\begin{proof}
Since $w(t,\cdot)\in X^*$ for every $t\in [0,T]$, the function
$$u(t,\cdot):=(-\Delta_{ND})^{-1}w(t,\cdot)\in \tilde X,\quad\text{for every $t\in [0,T]$}$$
is well-posed.
We want to prove $u(t,\cdot)\in C^1([0,T];\tilde X)$. By the pointwise definition of $u$ with respect to the variable $t$, we have
$$(\nabla u(t,\cdot),\nabla\varphi)_0=\langle w(t,\cdot)|\varphi\rangle_{X^*,X},\quad\text{for every $\varphi\in X, t\in [0,T]$}.$$
Since $w'(t,\cdot)\in X^*$, setting $z(t,\cdot):=(-\Delta_{ND})^{-1}w'(t,\cdot) \in \tilde X$ for every $t\in [0,T]$, we find that $z(t,\cdot)$ is the unique solution of
\begin{equation}\label{eqqq}
(\nabla z(t,\cdot),\nabla\varphi)_0=\langle w'(t,\cdot)|\varphi\rangle_{X^*,X},\quad\text{for every $\varphi\in X, t\in [0,T]$}.
\end{equation}
By the linearity and boundedness of the operator $(-\Delta_{ND})^{-1}$, we have $z(t,\cdot)-z(s,\cdot)=(-\Delta_{ND})^{-1}(w'(t,\cdot)-w(s,\cdot))$, for every $t,s\in [0,T]$, with the estimate
$$\|z(t,\cdot)-z(s,\cdot)\|_{\tilde X}\leq \|w'(t,\cdot)-w'(s,\cdot)\|_{X^*}.$$
Since $w'\in C^0([0,T]; X^*)$, then $z\in C^0([0,T]; \tilde X)$. Analogously, $u\in C^0([0,T];\tilde X)$.
Let $\xi\in C_0^\infty(0,T)$. Multiplying \eqref{eqqq} by $\xi$ and integrating over $[0,T]$, we obtain
\begin{align*}
\int_0^T (\nabla & z(t,\cdot), \nabla\varphi)_0\,\xi(t)\,\d t =\int_0^T \langle w'(t,\cdot)|\varphi\rangle_{X^*,X}\,\xi(t)\,\d t \\
&= \Big\langle \left(\int_0^T w'(t,\cdot)\,\xi(t)\,\d t \right) \Big|\varphi\Big\rangle_{X^*,X} = \Big\langle \left(-\int_0^T w(t,\cdot)\,\xi'(t)\,\d t \right) \Big|\varphi\Big\rangle_{X^*,X} \\
&= -\int_0^T \langle w(t,\cdot) |\varphi\rangle_{X^*,X}\,\xi'(t)\,\d t= -\int_0^T (\nabla u(t,\cdot),\nabla\varphi)_0\,\xi'(t)\,\d t.
\end{align*}
Now, since $z\in C^0([0,T]; \tilde X)$ and $\nabla :\tilde X \to \left(L^2(\Omega)\right)^n$ is a bounded linear operator (since $\|\nabla \phi\|_0=\|\phi\|_{\tilde X}$ for every $\phi\in \tilde X$), we can commute $\nabla$ with the integral $\int_0^T\cdot\,dt$, obtaining
%
%
$$\left(\nabla\int_0^T \Big(z(t,\cdot)\xi(t)+u(t,\cdot)\xi'(t)\Big)\,\d t, \nabla \varphi\right)_0=0\quad\text{for every $\varphi\in X, \xi\in C_0^\infty(0,T)$}.$$
Since $z(t,\cdot), u(t,\cdot)\in X$, then $\int_0^T z(t,\cdot)\xi(t)\,\d t, \int_0^T u(t,\cdot)\xi'(t)\,\d t\in X$, 
so, choosing $\varphi:=\int_0^T \Big(z(t,\cdot)\xi(t)+u(t,\cdot)\xi'(t)\Big)\,\d t\,\in X$, we have
$$\int_0^T \Big(z(t,\cdot)\xi(t)+u(t,\cdot)\xi'(t)\Big)\,\d t=0\quad\text{in $X$, for every $\xi\in C_0^\infty(0,T)$},$$
implying that $z=u'$.
Denoting by $\frac{\d}{\d t}$ the operator that associates a function to its weak time-derivative, we proved
$$(-\Delta_{ND})^{-1}\frac{\d}{\d t}=\frac{\d}{\d t}(-\Delta_{ND})^{-1}\quad\text{in $X^*$}$$
implying $u(t,\cdot)=(-\Delta_{ND})^{-1}w(t,\cdot)\in C^1([0,T];\tilde X)$. Eventually, by the estimates on $w$ in Proposition \ref{exunw}, we also have
\begin{align*}
\|u(t,\cdot)\|_{\tilde X}&=\|(-\Delta_{ND})^{-1}w(t,\cdot)\|_{\tilde X}\leq\|w(t,\cdot)\|_{X^*}\\
&\leq 2\max\{1,t\}\max\{1,c_\Omega\}\left(1+\|\sigma\|_{L^\infty(\Omega)}/\e\right)\|\tilde g\|_{1}.
\end{align*}
\begin{align*}
\|u'(t,\cdot)\|_{\tilde X}&=\|\frac{\d}{\d t}\Big((-\Delta_{ND})^{-1}w(t,\cdot)\Big)\|_{\tilde X}\\
&=\|(-\Delta_{ND})^{-1}w'(t,\cdot)\|_{\tilde X}\leq\|w'(t,\cdot)\|_{X^*}\\
&\leq 2\max\{1,c_\Omega\}\max\{1,\|\sigma\|_{L^\infty(\Omega)}/\e\}\left(1+\|\sigma\|_{L^\infty(\Omega)}/\e\right)\|\tilde g\|_{1},
\end{align*}
completing the proof.
\end{proof}

Combining Proposition \ref{solPBtime} with Proposition \ref{solPBstatic}, recalling that $\tilde X\subseteq X$ and using the \textit{Poincaré inequality}, we have proved Theorem \ref{Solution1}.

\medskip

\begin{remark}
By construction, $D(B)=X^*$. Moreover, by previous computations on $A$, we showed that $-B$ is a monotone operator and $R(Id_{X^*}+B)=X^*$. Thus, by Theorem $3.1$ in \cite{LumerPhillips61}, we find $-B$ generates (or $-B$ is the infinitesimal generator of) a strongly continuous contraction semigroup on $X^*$ denoted by $[0,+\infty[\,\ni t\mapsto T(t):=e^{-Bt}\in \mathcal{L}(X^*)$, namely, 
\begin{enumerate}
\item $T(0)=Id_{X^*}$;
\item $T(t+s)=T(t)T(s)$ for every $t,s\in\,[0,+\infty[$;
\item for every $\zeta\in X^*$, $\dsy\lim_{\,\,\,t\to 0^+}\|T(t)\zeta-\zeta\|_{X^*}=0$.
\end{enumerate}
Then, by Lemma $8.1$ in \cite{sinestrari}, we can represent the solution $w$ of \eqref{eq.BVPICweak} as
\begin{align*}
w(t,\cdot)&=e^{-Bt}(\Delta\tilde\omega)+\int_0^t e^{-B(t-s)}(-\Delta \omega)\,ds\\
&=e^{-Bt}(\Delta\tilde\omega)+\left(\int_0^t e^{-B(t-s)}\,ds\right)(-\Delta \omega)
\end{align*}
Recalling that $u(t,\cdot)=(-\Delta_{ND})^{-1}w(t,\cdot)$, we obtain
$$u(t,\cdot)=(-\Delta_{ND})^{-1}e^{-Bt}\left(\Delta\tilde\omega+\int_0^t e^{Bs}\,ds\,(-\Delta\omega)\right).$$
This representation formula makes explicit the linear dependence of $u$ on the initial data $\tilde\omega$ and on the static component $\omega$. Moreover, it shows that the time-dependent part $u$ is obtained by composing the resolvent $(-\Delta_{ND})^{-1}$ with the contraction semigroup generated by $-B$, which is the natural evolution operator associated with the degenerate parabolic structure of the problem.
\end{remark}

We can now prove the following energy estimate.

\begin{prop}
Under the assumptions of Proposition~\ref{solPBtime}, for every 
$0<\delta_0<\delta$ and every $t\in[0,T]$, the following estimate holds:
\begin{align*}
\|\nabla u(t)\|_{0}^{2}
&+\int_{0}^{t}\!\left(
\int_{(\sigma/\varepsilon>\delta)}\frac{\sigma(x)}{\varepsilon}\,|\nabla u(s,x)|^{2}\,dx
+\delta_{0}\!\int_{(\delta_{0}\le\sigma/\varepsilon\le\delta)}|\nabla u(s,x)|^{2}\,dx
\right)ds \\
&\le 
\max\!\left\{\frac{1}{\delta},\frac{1}{\delta_{0}}\right\}\, t\,\|\nabla\omega\|_{0}^{2}
+\|\nabla\tilde\omega\|_{0}^{2}.
\end{align*}
\end{prop}
\begin{proof}
Let $s\in\,]0,T]$ and let us set $a:=\sigma/\e$. We have
\begin{align*}
\frac{1}{2}\frac{\d}{\d s} \|\nabla u(s,\cdot)\|_0^2&=(\nabla u'(s,\cdot),\nabla u(s,\cdot))=\langle (-\Delta) u'(s,\cdot)|u(s,\cdot) \rangle_{X^*,X}\\
&=\langle w'(s,\cdot)|u(s,\cdot) \rangle_{X^*,X}.
\end{align*}
Using the weak formulation of Proposition \ref{solPBtime} with the test function $\varphi:=u(s,\cdot)\in \tilde X\subseteq X$, we obtain
\begin{equation}\label{Energiastima}
\frac{1}{2}\frac{\d}{\d s} \|\nabla u(s,\cdot)\|_0^2+\int_\Omega\,a(x)\langle\nabla u(s,x),\nabla u(s,x)\rangle\,dx=\int_\Omega\langle\nabla\omega,\nabla u(s,x)\rangle\,dx.
\end{equation}
Let $\delta>0$. We consider
$$\int_\Omega\langle\nabla\omega,\nabla u(s,x)\rangle\,dx=\int_{(a\leq \delta)}\langle\nabla\omega,\nabla u(s,x)\rangle\,dx+\int_{(a> \delta)}\langle\nabla\omega,\nabla u(s,x)\rangle\,dx.$$
Now, for $\delta_1>0$, by Cauchy-Schwarz and weighted Young inequalities, we have
\begin{align*}
\int_{(a\leq \delta)}\langle\nabla\omega,\nabla u(s,x)\rangle\,dx&\leq \frac{1}{2\delta_1}\int_{(a\leq \delta)}|\nabla\omega|^2\,dx+\frac{\delta_1}{2}\int_{(a\leq \delta)}|\nabla u(s,x)|^2\,dx.
\end{align*}
Analogously, for $\delta_2>0$,
\begin{align*}
\int_{(a> \delta)}\langle\nabla\omega,&\nabla u(s,x)\rangle\,dx=\int_{(a> \delta)}\frac{\sqrt{a}}{\sqrt{a}}\langle\nabla\omega,\nabla u(s,x)\rangle\,dx\\
&\leq \frac{1}{\sqrt{\delta}}\int_{(a> \delta)}\sqrt{a}\,\langle\nabla\omega,\nabla u(s,x)\rangle\,dx\\
&\leq \frac{1}{2\delta_2 \sqrt{\delta}}\int_{(a> \delta)}|\nabla\omega|^2\,dx+\frac{\delta_2}{2\sqrt{\delta}}\int_{(a> \delta)}a\,|\nabla u(s,x)|^2\,dx.
\end{align*}
Choosing $\delta_2=\sqrt{\delta}$ (so that $1-\frac{\delta_2}{2\sqrt{\delta}}=\frac{1}{2}$), putting in the left side of \eqref{Energiastima} the last term above and multiplying everything by $2$, we gain
\begin{align*}
\frac{\d}{\d s} \|\nabla u(s,\cdot)\|_0^2+&\int_{(a>\delta)}\,a |\nabla u(s,x)|^2\,dx+\int_{(a\leq\delta)}\,(2a-\delta_1) |\nabla u(s,x)|^2\,dx\\
&\leq \frac{1}{\delta_1}\int_{(a\leq  \delta)}|\nabla\omega|^2\,dx+\frac{1}{\delta}\int_{(a> \delta)}|\nabla\omega|^2\,dx.
\end{align*}
Let $0<\delta_0<\delta$. We can estimate
\begin{align*}
\int_{(a\leq\delta)}\,(2a-\delta_1) |\nabla u(s,x)|^2\,dx&\geq \int_{(\delta_0\leq a\leq \delta)}\,(2a-\delta_1) |\nabla u(s,x)|^2\,dx\\
&\geq (2\delta_0-\delta_1)\int_{(\delta_0\leq a\leq \delta)}\, |\nabla u(s,x)|^2\,dx
\end{align*}
Choosing now $\delta_1:=\delta_0$, we have
\begin{align*}
\frac{\d}{\d s} \|\nabla u(s,\cdot)\|_0^2+&\int_{(a>\delta)}\,a |\nabla u(s,x)|^2\,dx+\delta_0\int_{(\delta_0\leq a\leq\delta)}\,|\nabla u(s,x)|^2\,dx\\
&\leq \max\left\{\frac{1}{\delta_1},\frac{1}{\delta}\right\}\|\nabla\omega\|^2_0
\end{align*}
Integrating now on $s\in [0,t]$ and recalling that $u(0,\cdot)=\tilde\omega$, we complete the proof.
\end{proof}

\begin{coro}
For every $\delta>0$,
$$\limsup_{t\to+\infty}\,\frac{1}{t} \int_0^t \left(\int_{(\sigma/\e\,>\delta)} |\nabla u(s,x)|^2\,d x\right)\,ds\leq \frac{1}{\delta^2} \|\nabla\omega\|_{0}^2.$$
\end{coro}

\bigskip

\section{Regularity results}\label{regularity}

In this section, we establish regularity results for both the static component $\omega(x)$ and the dynamic component $u(t,x)$. We recall that the electrical conductivity $\sigma$ determines the geometry of the device. To motivate the following analysis, we consider the physically relevant detector configuration described in Section \ref{realcase} (see also Fig.\ref{fig:repr}). In that case, the conductivity $\sigma$ can be modeled as a piecewise constant function as follows:
\begin{equation*}
\left\{
\begin{array}{ll}
\sigma = 0, & \text{in } \overline{\Omega} \setminus \left( C_+ \cup \displaystyle\bigcup_{i=1}^{m_e} C_-^i \right), \\[4pt]
\sigma = c_0, & \text{in } C_+, \\[4pt]
\sigma = c_i, & \text{in } C_-^i,\quad i = 1, \ldots, m_e,
\end{array}
\right.
\end{equation*}
where $c_i > 0$ are constants, and the electrode regions $C_+,C_-^i$ are compact subsets of positive measure, with $d(\de C_+,\de C_-^i)>0, \,d(\de C_-^i,\de C_-^j)>0$ for every $i,j=1,\ldots,m_e, i\neq j$. This setting naturally leads us to focus on regions where the conductivity is constant.

\medskip
In order to obtain global regularity results, we recall that, as pointed out in Chapter $5$ in \cite{adams}, the uniform $C^d$-regularity property with $d\geq 1$ implies the segment property (see the precise definition in Chapter $5$ in \cite{adams} or in Section $9.8.1$ in \cite{DiBenedettoGianazza23}).

\medskip

Let us recall the following regularity results proved in \cite{DiBenedettoGianazza23} and \cite{evans}.

\begin{prop}[see Theorem $2$-$3$ in Section $6.3.1$ in \cite{evans}]\label{Hmregularity}
Let $\Omega'\subseteq\R^n$ be an open bounded set.
Let $u\in H^1(\Omega')$ be a weak solution of $-\Delta u=f$ in $\Omega'$ with $f\in H^m(\Omega')$ for some $m\in\N$. Then, $u\in H^{m+2}_{loc}(\Omega')$ and, for every compact set $K\subset \Omega'$, there exists a constant $c=c(m,\Omega',K)>0$ such that
$$\|u\|_{H^{m+2}(K)}\leq c\big(\|f\|_{H^m(\Omega')}+\|u\|_{L^2(\Omega')}\big).$$
Moreover, if $f\in C^\infty(\Omega')$, then $u\in C^\infty(\Omega')$.
\end{prop}

\begin{prop}[see Proposition $11.1$ in Section $2$ in \cite{DiBenedettoGianazza23}]\label{strongLapl}
Let $\Omega'\subseteq\R^n$ be an open bounded set.
Let $u\in C^2(\Omega')$ be a solution of $-\Delta u=f$ in $\Omega'$ with $f\in C^{m,\alpha}_{loc}(\Omega')$ for some $m\in\N, \alpha\in\,[0,1[$. Then, $u\in C^{m+2,\alpha}_{loc}(\Omega')$ and, for every compact sets $K\subset K' \subset \Omega'$, there exists a constant $c=c(n,m,d(K,\de K'))>0$ such that
$$\|u\|_{C^{m+2,\alpha}(K)}\leq c\big(\|f\|_{C^{m,\alpha}(K')}+\|u\|_{L^\infty(K')}\big).$$
\end{prop}

\begin{prop}[see Proposition $11.2$ in Section $2$ in \cite{DiBenedettoGianazza23}]\label{weakLapl}
Let $\Omega'\subseteq\R^n$ be an open bounded set.
Let $u\in L^1_{loc}(\Omega')$ be a local weak solution of $-\Delta u=f$ in $\Omega'$ with $f\in L^p_{loc}(\Omega')$ for some $p>n$. Then, there exists a constant $c=c(n,p)>0$ such that, for every balls $B(x_0,ar)\subset B(x_0,r)\subset \Omega'$, 
\begin{align*}
\|\nabla u\|_{L^\infty(B(x_0,ar))}&\leq c\|f\|_{L^p(B(x_0,r))}+c((1-a)r)^{-(n+1)}\|u\|_{L^1(B(x_0,r))}\\
\|u\|_{L^\infty(B(x_0,ar))}&\leq c r^{2-n/p}\|f\|_{L^p(B(x_0,r))}+c((1-a)r)^{-n}\|u\|_{L^1(B(x_0,r))}.
\end{align*}
\end{prop}

\begin{prop}[see Theorem $18.1$ in Section $9$ in \cite{DiBenedettoGianazza23}]\label{holderLapl}
Let $\Omega'\subseteq\R^n$ be an open set with the uniform $C^d$-regularity property with $d\geq 1$, and let $A=(a_{ij})$ be a $n\times n$ symmetric matrix with $a_{ij}\in L^\infty(\Omega)$, satisfying the elliptic condition
$$\lambda |\xi|^2\leq \sum_{i,j}a_{ij}(x)\xi_i \xi_j\leq \Lambda |\xi|^2,$$
for every $\xi\in\R^n, x\in\Omega$ for some $0<\lambda\leq \Lambda$.
Let $u\in W^{1,2}_{loc}(\Omega')$ be a local weak solution of 
$$-\div(A \nabla u)=f\in L^{\frac{n+\rho}{2}}(\Omega')$$
for some $\rho>0$. Then $u$ is locally $\alpha$-H\"older continuous in $\Omega'$ for some $\alpha=\alpha(n,f,A)\in\,]0,1[$, and, for every compact set $K\subset \Omega'$, there exist a constant $c=c(n,f,A,d(K,\de\Omega'))>0$ such that $\|u\|_{C^{0,\alpha}(K)}\leq c.$
\end{prop}

\bigskip

From now on, we set
$$\Lambda_{\e,\sigma}:=\left(1+\frac{\|\sigma\|_{L^\infty(\Omega)}}{\e}\right).$$

We are now in position to prove the following regularity results.

\begin{prop}\label{Regularityomega}
Assume Hypothesis \ref{Hgsigmae}.
Let $V(t,x)=\omega(x)+u(t,x)$ be the unique weak solution of \eqref{eq.BVPIC} as in Definition \ref{solutionBVP1}. We have the following:
\begin{enumerate}
\item $\omega$ is locally H\"older continuous in $\Omega$, and, for every compact set $K\subset \Omega$, there exist a constant $c=c(n,d(K,\de\Omega),\sigma,\e)>0$ and $\alpha=\alpha(n,\sigma,\e)\in\,]0,1[$ such that
$$\|\omega\|_{C^{0,\alpha}(K)}\leq c.$$
\item if $\Omega'\subseteq\Omega$ is an open bounded set such that $\sigma\equiv\sigma_1\geq 0$ is constant in $\Omega'$, then $\omega\in C^\infty(\Omega')$. Moreover, for every $m\in\N, \gamma\in [0,1[$ and every compact set $K\subset\Omega'$, there exists a constant $C=C(n,m,\Omega,K)>0$ such that
\begin{equation}\label{stimaomega}
\|\omega\|_{C^{m,\gamma}(K)}\leq C\,\Lambda_{\e,\sigma}\|\tilde g\|_{H^{1}(\Omega)}.
\end{equation}
\end{enumerate}
\end{prop}
\begin{proof}
\underline{Proof of $1)$}: By Proposition \ref{solPBstatic}, we know that $\omega\in H^1(\Omega)$ satisfies \eqref{equazione1}. So, in particular, since $C_0^\infty(\Omega)\subseteq X$, then $\omega$ satisfies
$$-\div((a+1)\nabla\omega)=0\quad\text{in $\mathcal{D}'(\Omega)$}.$$
Since $a\in L^\infty(\Omega)$ and $a(x)+1\geq 1$ for every $x\in\Omega$ (recall that $\sigma$ is non-negative), the above problem is uniformly elliptic with $L^\infty$ coefficients. By Proposition \ref{holderLapl}, we obtain that $\omega$ is locally H\"older in $\Omega$ with the estimate as in the statement.

\medskip
\underline{Proof of $2)$}: 
If $\sigma\equiv\sigma_1\geq 0$ is constant in $\Omega'$, then $\omega$ satisfies $-\Delta \omega=0$ in $\mathcal{D}'(\Omega')$. By the hypoellipticity of $-\Delta$, we have $\omega\in C^\infty(\Omega')$. 
In order to prove \eqref{stimaomega}, let $m\in\N, \gamma\in [0,1[\,, 0<a<a'<1,r>0, x_0\in\Omega$ such that $B(x_0,ar)\subset B(x_0,a'r)\subset B(x_0,r)\Subset\Omega$. Combining Proposition \ref{strongLapl} and Proposition \ref{weakLapl}, we have
\begin{align*}
\|\omega\|_{C^{m+2,\gamma}(B(x_0,ar))}&\leq c_{m,a'-a}\|\omega\|_{L^\infty(B(x_0,a'r))}\\
&\leq \tilde c_{n,m,a'-a}((1-a')r)^{-n}\|\omega\|_{L^1(B(x_0,r))}.
\end{align*}
Using a covering argument, we obtain that, for every compact set $K\subset\Omega$,
\begin{align*}
\|\omega\|_{C^{m+2,\gamma}(K)}&\leq c_{n,m,K}\|\omega\|_{L^1(K')}\leq \tilde c_{n,m,\Omega,K}\|\omega\|_{H^1(\Omega)}\\
&\leq C_{n,m,\Omega,K}(1+\|\sigma\|_{L^\infty(\Omega)}/\e)\|\tilde g\|_{H^{1}(\Omega)}, 
\end{align*}
where $K':=\{x\in\Omega'\,\,|\,\,d(x,K)\leq d(K,\de\Omega')/2\}$, and, in the last inequality, we used the estimate in Proposition \ref{solPBstatic}.
\end{proof}

\medskip
\begin{teo}
Assume Hypothesis \ref{Hgsigmae}.
Let $V(t,x)=\omega(x)+u(t,x)$ be the unique weak solution of \eqref{eq.BVPIC} as in Definition \ref{solutionBVP1}. Then, if $\Omega_0\subseteq \Omega$ is an open set such that $\sigma\equiv 0$ is constant in $\Omega_0$, then
\begin{enumerate}[(i)]
\item $u'\in C^0([0,T];C^{\infty}(\Omega_0))$. Moreover, for every $m\in\N, \gamma\in [0,1[$ and every compact set $K\subset \Omega_0$, there exists a constant $c=c(n,m,\Omega,K)>0$ such that
\begin{equation}\label{stimauprimo}
\|u'\|_{C^0([0,T];C^{m,\gamma}(K))}\leq c\, \Lambda_{\e,\sigma}^2\|\tilde g\|_{H^1(\Omega)}.
\end{equation}
\item if $\tilde g\in H^k_{loc}(\Omega_0)\cap H^2(\Omega)$ for some $k\geq 3$, then 
$$u\in C^1([0,T];H^k_{loc}(\Omega_0))\cap W^{1,\infty}(0,T; H^2(\Omega_0)).$$
Moreover, for every compact set $K\subset \Omega_0$, there exists a constant $c_1=c_1(n,k,\Omega,K)>0$ such that
\begin{gather}
\begin{split}
\|u\|_{C^{1}([0,T];H^{k}(K))}&\leq c_1\,\max\{1,T\}\Lambda_{\e,\sigma}^2\Big(\|\tilde g\|_{H^k(K)}+\|\tilde g\|_{H^1(\Omega)}\Big).
\end{split}
\end{gather}
\item if $\tilde g\in C^{k,\alpha}_{loc}(\Omega_0)\cap H^2(\Omega)$ for some $k\in\N, \alpha\in [0,1[$, then 
$$u\in C^1([0,T];C^{k,\alpha}_{loc}(\Omega_0)\cap H^2_{loc}(\Omega_0))\cap W^{1,\infty}(0,T; H^2(\Omega_0)).$$
Moreover, for every compact set $K\subset \Omega_0$, there exists a constant $c_2=c_2(n,k,\Omega,K)>0$ such that
\begin{gather}\label{ahnonlosoio}
\begin{split}
\|u\|_{C^{1}([0,T];C^{k,\alpha}(K))}&\leq c_2\,\max\{1,T\}\Lambda_{\e,\sigma}^2\Big(\|\tilde g\|_{C^{k,\alpha}(K)}+\|\tilde g\|_{H^1(\Omega)}\Big),\\
\|u\|_{C^{1}([0,T];H^{2}(K))}&\leq c_2\,\max\{1,T\}\Lambda_{\e,\sigma}^2\Big(\|\tilde g\|_{H^2(K)}+\|\tilde g\|_{H^1(\Omega)}\Big).
\end{split}
\end{gather}
\end{enumerate}
\end{teo}
\begin{proof}
\noindent{\underline{Proof of $(i)$}}: Let $\Omega_0\subseteq \Omega$ be an open set such that $\sigma\equiv 0$ in $\Omega_0$. Then, by Proposition \ref{solPBtime}, we have
$$\langle w'(t,\cdot)|\varphi\rangle_{X^*,X}=\int_{\Omega_0}\langle\nabla\omega,\nabla\varphi\rangle\,dx,\quad\text{for every $\varphi\in C_0^\infty(\Omega_0)\,(\subseteq X)$},\,t\in [0,T]$$
i.e.\,$w'(t,\cdot)=-\Delta\omega$ in $\mathcal{D}'(\Omega_0)$ for every $t\in [0,T]$. Since $u'(t,\cdot)=(-\Delta_{ND})^{-1}w'(t,\cdot)$, we have
$$-\Delta u'(t,\cdot)= w'(t,\cdot)=-\Delta\omega\quad\text{in $\mathcal{D}'(\Omega_0)$\,\,\,\,for every $t\in [0,T]$}.$$
By Proposition \ref{Regularityomega}, $-\Delta\omega\in C^\infty(\Omega_0)$. Thus, by the hypoellipticity of $-\Delta$, we have $u'(t,\cdot)\in C^\infty(\Omega_0)$ for every $t\in [0,T]$.

Moreover, by Proposition \ref{strongLapl}, Proposition \ref{Regularityomega} and by the estimate in Proposition \ref{solPBtime}, for every $m\in\N, \gamma\in [0,1[\,, 0<a<a'<1,r>0, x_0\in\Omega$ such that $B(x_0,ar)\subset B(x_0,a'r)\subset B(x_0,r)\Subset\Omega_0$, 
\begin{align*}
&\|u'(t,\cdot)\|_{C^{m+2, \gamma}(B(x_0,ar))}\leq c_{m,a'-a}\big(\|-\Delta\omega\|_{C^{m, \gamma}(B(x_0,a'r))}\\
&\quad\quad\quad\quad\quad\quad\quad\quad\quad\quad+\|u'(t,\cdot)\|_{L^\infty(B(x_0,a'r))}\big)\\
&\quad\quad\quad\quad\leq \tilde c_{n,m,a'-a}\big(\|\omega\|_{C^{m+2, \gamma}(B(x_0,a'r))}+r^{2}\|-\Delta\omega\|_{L^{\infty}(B(x_0,r))}\\
&\quad\quad\quad\quad\quad\quad\quad\quad\quad\quad+((1-a')r)^{-n}\|u'(t,\cdot)\|_{L^1(B(x_0,r))}\big)\\
&\quad\quad\quad\quad\leq C_{n,m,a'-a,r}\Big((1+r^{2})\|\omega\|_{C^{m+2, \gamma}(B(x_0,a'r))}\\
&\quad\quad\quad\quad\quad\quad\quad\quad\quad\quad+((1-a')r)^{-n}\|u'(t,\cdot)\|_{H^1(B(x_0,r))}\Big).
\end{align*}
Applying the estimate for $\omega$ in Proposition \ref{Regularityomega} and the estimate for $u$ in Proposition \ref{solPBtime} to previous computations, we obtain
$$\|u'(t,\cdot)\|_{C^{m+2,\gamma}(B(x_0,ar))}\leq C_{n,m,\Omega,a'-a,r}(1+\|\sigma\|_{L^\infty(\Omega)}/\e)^2\|\tilde g\|_{H^1(\Omega)}.$$
By a covering argument as in Proposition \ref{Regularityomega}, we proved, for every compact set $K\subset \Omega_0$ and every $m\in\N,\gamma\in [0,1[$\,, 
\begin{equation}\label{stimaSUPu}
\|u'(t,\cdot)\|_{C^{m+2,\gamma}(K)}\leq c_{n,m,\Omega,K}(1+\|\sigma\|_{L^\infty(\Omega)}/\e)^2\|\tilde g\|_{H^1(\Omega)}, \quad\text{$\forall\,t\in [0,T]$}.
\end{equation}
We have proved $u'\in L^\infty(0,T;C^\infty(\Omega_0))$, showing the estimate \eqref{stimauprimo}.

We want to prove $u'\in C^0([0,T];C^\infty(\Omega_0))$. We observe 
$$-\Delta(u'(t,\cdot)-u'(s,\cdot))=0\quad\text{in $\mathcal{D}'(\Omega_0)$\quad for every $t,s\in [0,T]$}.$$
Thus, arguing as above, we have, for every $m\in\N,\gamma\in [0,1[$ and $K\subset\Omega_0$ compact set
$$\|u'(t,\cdot)-u'(s,\cdot)\|_{C^{m+2,\gamma}(K)}\leq C_{n,m,\Omega,K}\|u'(t,\cdot)-u'(s,\cdot)\|_{L^1(K')},$$
where $K':=\{x\in\Omega_0\,\,|\,\,d(x,K)\leq d(K,\de\Omega_0)/2\}\subset \Omega_0$. Since $u'\in C^0([0,T];H^1(\Omega))$, we find
$$\|u'(t,\cdot)-u'(s,\cdot)\|_{L^1(K')}\leq C_{\Omega}\|u'(t,\cdot)-u'(s,\cdot)\|_{H^1(\Omega)}\xrightarrow{t\to s} 0,$$
showing that $u'\in C^0([0,T];C^\infty(\Omega_0))$.

\medskip
\noindent{\underline{Proof of $(iii)$}}:
Let $\tilde g\in C^{k,\alpha}_{loc}(\Omega_0)\cap H^2(\Omega)$ for some $k\in\N,\alpha\in [0,1[$. By construction, we have $u(0,\cdot)=-\tilde\omega=\omega-\tilde g$. By Proposition \ref{Regularityomega}, we find $u(0,\cdot)\in C^{k,\alpha}_{loc}(\Omega_0)\cap H^2(\Omega_0)$. By the fundamental theorem of calculus applied to $u$ in $H^1(\Omega)$ (see Theorem $2$ in Section $5.9.2$ in \cite{evans}), we know
\begin{equation}\label{TFCIu}
u(t,\cdot)= -\tilde\omega+\int_0^t u'(s,\cdot)\,ds \quad\text{in $H^1(\Omega)$}.
\end{equation}
Using the regularity of $\tilde\omega$ and $u'(t,\cdot)$, we have $u(t,\cdot)\in C^{k,\alpha}_{loc}(\Omega_0)\cap H^2(\Omega_0)$ for every $t\in [0,T]$. Thus, up to considering modifications on sets of measure zero, \eqref{TFCIu} holds pointwise in $x\in\Omega$. By the estimate \eqref{stimaSUPu} which is uniform in $t\in [0,T]$, for every compact set $K\subset\Omega_0$, we have
\begin{align*}
\|u(t,\cdot)\|_{C^{k,\alpha}(K)}&\leq \|\tilde\omega\|_{C^{k,\alpha}(K)}+\int_0^t\|u'(s,\cdot)\|_{C^{k,\alpha}(K)}\,ds\\
&\leq \|\omega\|_{C^{k,\alpha}(K)}+\|\tilde g\|_{C^{k,\alpha}(K)}+ t\, C_{n,k,\Omega,K}(1+\|\sigma\|_{L^\infty(\Omega)}/\e)^2\|\tilde g\|_{H^1(\Omega)}\\
&\leq \tilde C_{n,k,\Omega,K}\max\{1,t\}(1+\|\sigma\|_{L^\infty(\Omega)}/\e)^2\Big(\|\tilde g\|_{C^{k,\alpha}(K)}+\|\tilde g\|_{H^1(\Omega)}\Big)
\end{align*}
and
\begin{align*}
\|u(t,\cdot)&\|_{H^{2}(K)}\leq \|\tilde\omega\|_{H^{2}(K)}+\int_0^t\|u'(s,\cdot)\|_{H^{2}(K)}\,ds\\
&\leq \|\omega\|_{C^{2}(K)}+\|\tilde g\|_{H^{2}(K)}+\int_0^t\|u'(s,\cdot)\|_{C^{2}(K)}\,ds\\
&\leq c_{n,\Omega,K}\max\{1,t\}(1+\|\sigma\|_{L^\infty(\Omega)}/\e)^2\Big(\|\tilde g\|_{H^2(K)}+\|\tilde g\|_{H^1(\Omega)}\Big).
\end{align*}
We have proved
$u\in W^{1,\infty}(0,T;C^{k,\alpha}_{loc}(\Omega_0)\cap H^2(\Omega_0))$ with the estimates
\begin{gather}\label{stimaW1infty}
\begin{split}
\|u\|_{W^{1,\infty}(0,T;C^{k,\alpha}(K))}&\leq c_{n,k,\Omega,K}\max\{1,T\}\Lambda_{\e,\sigma}^2\Big(\|\tilde g\|_{C^{k,\alpha}(K)}+\|\tilde g\|_{H^1(\Omega)}\Big),\\
\|u\|_{W^{1,\infty}(0,T;H^2(K))}&\leq c_{n,\Omega,K}\max\{1,T\}\Lambda_{\e,\sigma}^2\Big(\|\tilde g\|_{H^2(K)}+\|\tilde g\|_{H^1(\Omega)}\Big),
\end{split}
\end{gather}
for every compact set $K\subset\Omega_0$, showing also the estimate \eqref{ahnonlosoio}.
Now, using again \eqref{TFCIu}, we have
$$u(t,x)-u(s,x)=\int_s^t u'(\tau,x)\,d\tau\quad\text{for every $x\in\Omega$}.$$
Hence, by \eqref{stimauprimo},
\begin{align*}
\|u(t,\cdot)-u(s,\cdot)\|_{C^{k,\alpha}(K)}&\leq |t-s|\,\sup_{t\in [0,T]}\|u'(t,\cdot)\|_{C^{k,\alpha}(K)}\\
&\leq |t-s|\,c_{n,k,\Omega,K}(1+\|\sigma\|_{L^\infty(\Omega)}/\e)^2\|\tilde g\|_{H^1(\Omega)}\\
&=c_{n,k,\Omega,K,\tilde g,\sigma,\e}|t-s|\xrightarrow{t\to s} 0,
\end{align*}
and
\begin{align*}
\|u(t,\cdot)-u(s,\cdot)\|_{H^2(K)}&\leq |t-s|\,\sup_{t\in [0,T]}\|u'(t,\cdot)\|_{H^2(K)}\\
&\leq |t-s|\,\sup_{t\in [0,T]}\|u'(t,\cdot)\|_{C^2(K)}\\
&\leq |t-s|\,c_{n,\Omega,K}(1+\|\sigma\|_{L^\infty(\Omega)}/\e)^2\|\tilde g\|_{H^1(\Omega)}\\
&=c_{n,\Omega,K,\tilde g,\sigma,\e}|t-s|\xrightarrow{t\to s} 0,
\end{align*}
showing that $u\in C^1([0,T];C^{k,\alpha}_{loc}(\Omega_0)\cap H^2_{loc}(\Omega_0))$.

\medskip
\noindent{\underline{Proof of $(ii)$}}:
Let $\tilde g\in H^k_{loc}(\Omega_0)\cap H^2(\Omega)$ for some $k\geq 3$. The proof of $(ii)$ can be made analogously as in the previous point to show that $u\in C^1([0,T]; H^2_{loc}(\Omega_0))$.
\end{proof}
\medskip

\begin{teo}
Assume Hypothesis \ref{Hgsigmae}.
Let $V(t,x)=\omega(x)+u(t,x)$ be the unique weak solution of \eqref{eq.BVPIC} as in Definition \ref{solutionBVP1}. Then, if $\Omega_1\subseteq \Omega$ is an open set such that $\sigma\equiv \sigma_1>0$ is constant in $\Omega_1$, then
\begin{enumerate}[(i)]
\item if $\tilde g\in H^k_{loc}(\Omega_1)\cap H^2(\Omega)$ for some $k\geq 3$, then 
$$u\in C^1([0,1];H^k_{loc}(\Omega_1))\cap W^{1,\infty}(0,T; H^2(\Omega_1)).$$ 
Moreover, for every compact set $K\subset \Omega_1$, there exists a constant $c_1=c_1(n,k,\Omega,K)>0$ such that
\begin{gather}\label{stimauC1secondo}
\begin{split}
\|u\|_{C^{1}([0,T];H^{k}(K))}&\leq c_1\max\{1,T\}\left(1+\frac{\sigma_1}{\e}\right)\Lambda_{\e,\sigma}^2\Big(\|\tilde g\|_{H^k(K)}+\|\tilde g\|_{H^1(\Omega)}\Big).
\end{split}
\end{gather}
\item if $\tilde g\in C^{k,\alpha}_{loc}(\Omega_1)\cap H^2(\Omega)$ for some $k\in\N, \alpha\in [0,1[$, then 
$$u\in C^1([0,1];C^{k,\alpha}_{loc}(\Omega_1)\cap H^2_{loc}(\Omega_1))\cap W^{1,\infty}(0,T; H^2(\Omega_1))$$ 
Moreover, for every compact set $K\subset \Omega_1$, there exists a constant $c_2=c_2(n,k,\Omega,K)>0$ such that
\begin{gather}\label{stimauC1}
\begin{split}
\|u\|_{C^{1}([0,T];C^{k,\alpha}(K))}&\leq c_2\,\max\{1,T\}\left(1+\frac{\sigma_1}{\e}\right)\Lambda_{\e,\sigma}^2\Big(\|\tilde g\|_{C^{k,\alpha}(K)}+\|\tilde g\|_{H^1(\Omega)}\Big),\\
\|u\|_{C^{1}([0,T];H^{2}(K))}&\leq c_2\,\max\{1,T\}\left(1+\frac{\sigma_1}{\e}\right)\Lambda_{\e,\sigma}^2\Big(\|\tilde g\|_{H^2(K)}+\|\tilde g\|_{H^1(\Omega)}\Big).
\end{split}
\end{gather}
\end{enumerate}
\end{teo}
\begin{proof}
Let $\Omega_1\subseteq \Omega$ be an open set such that $\sigma\equiv \sigma_1>0$ is constant in $\Omega_1$. Then, by Proposition \ref{solPBtime}, we have
$$\langle w'(t,\cdot)|\varphi\rangle_{X^*,X}+\frac{\sigma_1}{\e}\int_{\Omega_1}\langle \nabla(-\Delta_{ND})^{-1} w(t,\cdot),\nabla\varphi(t,\cdot)\rangle\,dx=\int_{\Omega_1}\langle\nabla\omega,\nabla\varphi\rangle\,dx,$$
for every $\varphi\in C_0^\infty(\Omega_1)\,$, $t\in [0,T]$.
Thus, recalling that $u(t,\cdot)=(-\Delta_{ND})^{-1}w(t,\cdot)$, we have 
$$-\Delta\Big(u'(t,\cdot)+\frac{\sigma_1}{\e} u(t,\cdot)\Big)=-\Delta\omega\in C^\infty(\Omega_1)\quad\text{in $\mathcal{D}'(\Omega_1)$,\,\,\, $\forall\,t\in [0,T]$}.$$ Let us set $f(t,\cdot):=u'(t,\cdot)+\frac{\sigma_1}{\e} u(t,\cdot)$. Arguing as for the proof of \eqref{stimaSUPu}, we obtain $f(t,\cdot)\in C^\infty(\Omega_1)$ for every $t\in [0,T]$, and 
\begin{equation}\label{stimaSUPf}
\|f(t,\cdot)\|_{C^{m+2,\gamma}(K)}\leq c_{n,m,\Omega,K}(1+\|\sigma\|_{L^\infty(\Omega)}/\e)^2\|\tilde g\|_{H^1(\Omega)},\quad\text{$\forall\,t\in [0,T]$},
\end{equation}
for any compact set $K\subset \Omega_1$ and $m\in\N,\gamma\in [0,1[$. In particular, we have $f\in L^\infty(0,T; C^\infty(\Omega_1))$. Moreover, since $-\Delta(f(t,\cdot)-f(s,\cdot))=0$ in $\mathcal{D}'(\Omega_1)$ for every $t,s\in [0,T]$, arguing again as above, we find $f\in C^0([0,T]; C^\infty(\Omega_1))$.
\medskip

\underline{Proof of $(ii)$}. Let $\tilde g\in C^{k,\alpha}_{loc}(\Omega_1)\cap H^2(\Omega)$ for some $k\in\N,\alpha\in [0,1[$. Let us define the function $h:[0,T]\times \Omega_1\to \R$ such that 
\begin{equation}\label{defh}
h(t,x):=e^{-\frac{\sigma_1}{\e}t}(-\tilde\omega)+\int_0^t e^{-\frac{\sigma_1}{\e}(t-s)}f(s,x)\,ds.
\end{equation}
Using the regularity of $f$ and $\tilde\omega$, we have $h\in C^1([0,T];C_{loc}^{k,\alpha}(\Omega_1)\cap H^2_{loc}(\Omega_1))$ and $h(t,\cdot)\in H^2(\Omega_1)$ for every $t\in [0,T]$. Moreover, $h$ satisfies the following Cauchy problem in $L^2(\Omega_1)$
\begin{equation}\label{ODEcauchy}
\left\{
\begin{array}{ll}
 y'(t)=-\frac{\sigma_1}{\e}y(t)+f(t),&\quad\text{in $L^2(\Omega_1)$, $t\in [0,T]$}, \\
y(0) = -\tilde\omega\in L^2(\Omega_1).&
\end{array}\right.
\end{equation}
We observe that $-\frac{\sigma_1}{\e} Id_{L^2(\Omega_1)}:L^2(\Omega_1)\to L^2(\Omega_1)$ is trivially a Hille-Yosida operator, $f\in C^0([0,T]; C^0(\Omega_1))\subseteq L^2(0,T; L^2(\Omega_1))$ and $h\in C^1([0,T];C^0(\Omega_1))\subseteq W^{1,2}(0,T;L^2(\Omega_1))$. Thus, by Theorem \ref{teoLpsolution}, we have that $h$ is the unique $L^2$-solution of \eqref{ODEcauchy}.
On the other hand, by Proposition \ref{solPBtime}, we know that $u\in C^1([0,T];\tilde X)\subseteq W^{1,2}(0,T;L^2(\Omega_1))$ and $u$ satisfies \eqref{ODEcauchy} by the definition of $f$. By the uniqueness of the solution of \eqref{ODEcauchy}, we obtain $u=h$, showing that $u\in C^1([0,T];C^{k,\alpha}_{loc}(\Omega_1)\cap H^2_{loc}(\Omega_1))$ and $u(t,\cdot)\in H^2(\Omega_1)$ for every $t\in [0,T]$.

We observe that, by the definition of $f$ and by \eqref{defh}, $u'(t,x)=-\frac{\sigma_1}{\e}u(t,x)+f(t,x)$ and $u(t,x)=e^{-\frac{\sigma_1}{\e}t}(-\tilde\omega)+\int_0^t e^{-\frac{\sigma_1}{\e}(t-s)}f(s,x)\,ds$
for every $(t,x)\in [0,T]\times\Omega_1$ (up to consider modifications on sets of measure zero).
Using all the estimates we have proved before,
\begin{align*}
&\|u\|_{C^1([0,T];C^{k,\alpha}(K))}=\sup_{t\in [0,T]}\|u(t,\cdot)\|_{C^{k,\alpha}(K)}+\sup_{t\in [0,T]}\|u'(t,\cdot)\|_{C^{k,\alpha}(K)}\\
&\quad\quad\leq (1+\sigma_1/\e)\sup_{t\in [0,T]}\|u(t,\cdot)\|_{C^{k,\alpha}(K)}+\sup_{t\in [0,T]}\|f(t,\cdot)\|_{C^{k,\alpha}(K)}\\
&\quad\quad\leq 2(1+\sigma_1/\e)\big(\|\tilde\omega\|_{C^{k,\alpha}(K)}+\max\{1,T\}\sup_{t\in [0,T]}\|f(t,\cdot)\|_{C^{k,\alpha}(K)}\big)\\
&\quad\quad\leq 2(1+\sigma_1/\e)\big(\|\omega\|_{C^{k,\alpha}(K)}+\|\tilde g\|_{C^{k,\alpha}(K)}+\max\{1,T\}\sup_{t\in [0,T]}\|f(t,\cdot)\|_{C^{k,\alpha}(K)}\big)\\
&\quad\quad\leq c_{n,k,\Omega,K}(1+\sigma_1/\e)\max\{1,T\}(1+\|\sigma\|_{L^\infty(\Omega)}/\e)^2\Big(\|\tilde g\|_{C^{k,\alpha}(K)}+\|\tilde g\|_{H^1(\Omega)}\Big),
\end{align*}
for every compact set $K\subset\Omega_1$.

\underline{Proof of $(i)$}. Let $\tilde g\in H^{k}_{loc}(\Omega_1)\cap H^2(\Omega)$ for some $k\geq 3$. The proof can be made analogously, recalling that $\tilde g\in H^2(\Omega)$ so that $h\in  W^{1,\infty}(0,T;H^2(\Omega_1))\subseteq L^2(0,T;L^2(\Omega_1))$.
\end{proof}

\medskip

\begin{coro}\label{regularityTOT}
Let $\Omega\subseteq\R^3$ be the volume of detector described in Section \ref{realcase} (see Fig.\ref{fig:repr}) and assume Hypothesis \ref{Hgsigmae} with respect to the function $\tilde g$ as in \eqref{examplegtilde}. Let $V$ be the unique weak solution of \eqref{eq.BVPIC} as in Definition \ref{solutionBVP1}. If $\Omega_1\subseteq \Omega$ is an open set such that $\sigma\equiv \sigma_1\geq 0$ is constant in $\Omega_1$, then
$$V\in C^1([0,T];C^\infty(\Omega_1)).$$
Moreover, for every $k\in\N$ and for every $K\subset\Omega_1$ compact set, there exists a constant $c=c(n,k,\Omega,K)>0$ such that
$$\|V\|_{C^{1}([0,T];C^{k}(K))}\leq c\max\{1,T\}\left(1+\frac{\sigma_1}{\e}\right)\Lambda_{\e,\sigma}^2\Big(\|\tilde g\|_{C^{k}(K)}+\|\tilde g\|_{H^1(\Omega)}\Big).$$
\end{coro}

\bigskip
\section{An explicit example of 3D solid-state detector}\label{realcase}

The analysis and numerical solutions of the quasi-static Maxwell BVP \eqref{QSM1eq} are non-trivial due to the complex geometry of modern detectors, the discontinuity of the conductivity function $\sigma$ of the material, and the nature of the equation itself. Traditional numerical approaches have been used to analyze these effects, including the Finite Element Method (FEM), for instance implemented in COMSOL Multiphysics \cite{comsol}, and spectral methods \cite{Anderlini_2026}. More recently, Physics-Informed Neural Networks (PINNs) have also been proposed as an alternative framework for solving the governing equations \cite{bombini2025PINN, R_SIF25}, offering a mesh-free approach.
\medskip

To illustrate the physical relevance of the model, we consider a representative configuration of a 3D solid‑state diamond detector, following the setup analyzed in \cite{Anderlini_2026}.
Let $n=3$ and consider $\Omega$ as the interior of the cubic volume $\Pi_{j=1}^{3}\,[-\ell_j,\ell_j]\subseteq \R^3$ with edges of length $2\ell_j$, whose corners and edges are rounded on a sufficiently small scale \(\delta>0\) to ensure the regularity of the boundary (for example, considering $\delta$ smaller than the fabrication error of the components).
The polarized electrode is modeled as the cylindrical region
$$C_+:=\{(x,y,z)\in\R^{3}\,\,:\,\,\|(x,y)\|_{\R^2}\leq \rho_+,\,\,L_+\leq z\leq \ell_3\},$$
where $0<L_+<2\ell_3$ denotes the height of the cylinder and $\rho_+>0$ is the radius of the base. The grounded electrodes $C_-^i$ are obtained by translating $C_+$ as follows: given a spacing parameter $a>\rho_+\sqrt{2}$, we consider $\{C^i_-\}_i:=\{\mathbf{v}+C_+\}$ where the translation vectors $\mathbf{v}$ range in
\begin{align*}
\mathbf{v}\,\,\in&\,\,\left\{\Big((d_1+1)a, (d_2+1)a,-(\ell_3-L_+)\Big)\,\,|\,\,d_j\in 2\Z,\,\,|d_j|\leq N,\,\,\text{for $j=1,2$}\right\}\\
&\cup \left\{(d_1a,d_2a,0)\,\,|\,\,d_j\in 2\Z,\,\,d_j\neq 0,\,\,|d_j|\leq N,\,\,\text{for $j=1,2$}\right\}
\end{align*}
for some integer $N\in\N$ chosen according to the portion of detector under consideration. In Fig.\,\ref{fig:repr}, there is a graphical representation of $\Omega$ described above.

In this configuration, the conductivity $\sigma$ of the diamond bulk is negligible compared to that of the graphitic electrodes (see \cite{Anderlini_2026}) and $\sigma$ can be regarded as constant within each electrode region. The polarized and grounded electrodes $C_+,C_-^i$ are compact subsets of $\overline{\Omega}$ of positive measure, and they are mutually separated, namely $d(\de C_+,\de C_-^i)\geq a\sqrt{2}-2\rho_+>0, \,d(\de C_-^i,\de C_-^j)\geq a\sqrt{2}-2\rho_+>0$ for every $i,j=1,\ldots,m_e,\,i\neq j$.

Moreover, since $n=3$ and $2-\frac{n}{2}=\frac{1}{2}>0$, by the Morrey's inequality, $\tilde g\in H^2(\Omega)\hookrightarrow C^{0,\gamma}(\overline{\Omega})$ for $0<\gamma\leq \frac{1}{2}$ and $g\in C^0(\de\Omega)$ is the restriction of $\tilde g$ in the usual sense. An explicit example of $\tilde g:\overline{\Omega}\to\R$ can be given as follows: if $\beta>0$ and $\delta\in\,]0,a/2[$ are parameters suitable chosen, then
\begin{equation}\label{examplegtilde}
\tilde g(x,y,z):=\left\{
\begin{array}{ll}
V_0\left(1-e^{-\frac{1}{\beta\,|\ell_3-z|^2}}\right)\,\eta(z)\,\chi(x,y),&\quad\text{if  $z<\ell_3$},\\
V_0\,\chi(x,y),&\quad\text{if $z=\ell_3$},
\end{array}\right.
\end{equation}
where $V_0>0$ is constant, $\chi\in C_0^\infty(B_{\rho_++\delta}((0,0)))$ is a cutoff function such that $0\leq \chi\leq 1$ and $\chi|_{B_{\rho_+}((0,0))}\equiv 1$ (we observe $\mathrm{Int}(\de C_+\cap\de\Omega)\cong B_{\rho_+}((0,0))\subseteq \R^2$), and $\eta\in C^\infty(\R)$ is an increasing function such that $0\leq \eta\leq 1$, $\eta\equiv 0$ if $z\leq z_0$ and $\eta\equiv 1$ if $z\geq z_1$ for some $-\ell_3<z_0<z_1<\ell_3$ suitable chosen. As a consequence, $\tilde g\in C^\infty(\overline{\Omega})$, 
$$g|_{\de\Omega\,\cap\,\de C_+}\equiv V_0,\quad\quad g|_{\de\Omega\,\cap\,\de C_-^i}\equiv 0\quad\text{for every $i=1,\ldots,m_e$},$$
as prescribed by the Ramo-Schockley theorem, and $\de_\nu \tilde g\equiv 0$ in $B_N$.

\begin{figure}[h!]
    \centering
    \includegraphics[width=0.6\linewidth]{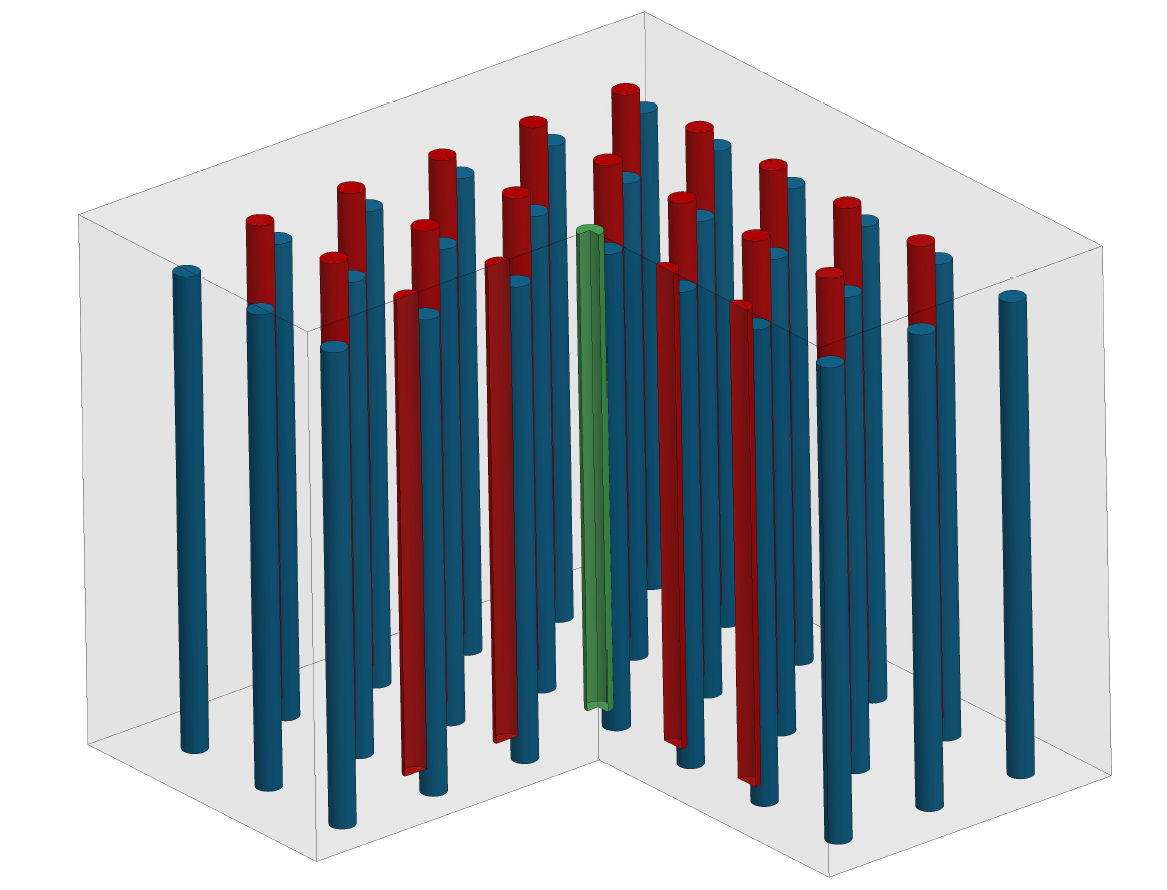}
    \caption{Schematic representation of a 3D diamond detector showing a portion of the sensor composed of multiple repeated cells. The illustration highlights the electrodes connected to the top (red), those to the bottom (blue) of the volume of the detector and the polarized one (green). The layout is indicative and not drawn to scale. See \cite{Anderlini_2026} for more details.}
    \label{fig:repr}
\end{figure}


\bigskip
\appendix
\section{Elliptic BVP with mixed Dirichlet-Neumann BC}\label{sec:A1}

Let $A\in L^\infty(\Omega,\R^{n\times n})$ be a symmetric uniformly positive matrix on $\Omega$, i.e.\,
$$\lambda_{min}(A)\,|\xi|^2\leq \langle A(x)\xi, \xi\rangle \leq \lambda_{max}(A)\,|\xi|^2,$$
for every $\xi\in\R^n$, for a.e.\,$x\in\Omega$, where $\lambda_{min}(A),\lambda_{max}(A)\in\,]0,+\infty[$ are positive constants. 

For every $F\in X^*$ (where $X^*$ is the dual space of $X$ defined in \eqref{defX}) we consider the \textit{elliptic BVP with mixed Dirichlet-Neumann boundary conditions} given by
\begin{equation}\label{eq.BVPNDgeneral}
\left\{
\begin{array}{ll}
\dsy -\div \big( A \nabla \eta \big) = F,&\quad\text{in $\Omega$}, \\
\eta = 0,&\quad\text{in $B_D$,} \\
\dsy\langle A\nabla\eta,\nu\rangle  = 0,&\quad\text{in $B_N$}.
\end{array}\right.
\end{equation}
Suppose $\eta$ is a regular solution of \eqref{eq.BVPNDgeneral} and $A\in C^\infty(\Omega,\R^{n\times n})$. Let $\varphi \in C^\infty(\overline{\Omega})$ be such that $\varphi|_{B_D}\equiv 0$. Then, multiplying the left side of the first equation in \eqref{eq.BVPNDgeneral} by $\varphi$, integrating by parts and using the boundary conditions, we obtain
\begin{align*}
\int_\Omega\,-\div\big(A\nabla \eta\big)\,\varphi\,dx &=-\int_{\de\Omega}\,\langle A\nabla\eta,\nu\rangle\, \varphi\,\d H^{n-1}+\int_\Omega\,\langle A\nabla\eta,\nabla \varphi\rangle\,dx\\
&=\int_\Omega\,\langle A\nabla\eta,\nabla \varphi\rangle\,dx.
\end{align*}
Hence, we consider the linear operator 
$$\langle-\div(A\nabla \eta)|\varphi\rangle_{X^*,X}:=(A\nabla \eta, \nabla\varphi)_0,\quad\text{for every $\varphi\in X$},$$
which is well-posed and bounded since
\begin{equation}\label{eq.stimaaND}
|(A\nabla \eta, \nabla\varphi)_0|\leq \|A\|_{L^\infty(\Omega,\R^{n\times n})}\|\nabla \eta\|_{0}\|\nabla\varphi\|_{0}= \lambda_{max}(A)\|\eta\|_{X}\|\varphi\|_{X}
\end{equation}
for every $\varphi\in X$.

\begin{defi}
We call a weak solution of \eqref{eq.BVPNDgeneral} a function $\eta\in X$ that satisfies the equation
\begin{equation}\label{eq.weakformBVPND}
(A\nabla\eta,\nabla\varphi)_0=\langle F|\varphi\rangle_{X^*,X},\quad\text{for every $\varphi\in X$}.
\end{equation}
\end{defi}

\begin{prop}
For every symmetric uniformly positive matrix $A\in L^\infty(\Omega,\R^{n\times n})$ and for every $F\in X^*$, there exists a unique weak solution $\eta\in X$ of \eqref{eq.BVPNDgeneral} such that 
\begin{equation}\label{eq.stimacontF}
\|\eta\|_{X}\leq \frac{\| F\|_{X^*}}{\lambda_{min}(A)}.
\end{equation}
\end{prop}
\begin{proof}
Let $a:X\times X\to \R$ be the bilinear map defined as
$$a(\eta,\varphi):=(A\nabla \eta, \nabla\varphi)_0,\quad\quad\text{for any $\eta,\varphi \in X$}.$$
Then, $a$ is continuous by \eqref{eq.stimaaND}. Moreover, $a$ is coercive. Indeed, 
$$a(\eta,\eta)=(A\nabla \eta, \nabla\eta)_0\geq \lambda_{min}(A)\,\|\eta\|^2_{X}.$$
Since $F \in X^*$, by the \textit{Lax-Milgram Theorem}, there exists a unique weak solution $\eta\in X$ of problem \eqref{eq.BVPNDgeneral} such that the estimate \eqref{eq.stimacontF} holds.
\end{proof}

\bigskip

Let us define the solution operator of \eqref{eq.BVPNDgeneral} with respect to the Laplacian $-\Delta$ (i.e.\,$A=Id_{n\times n}$) as
$$(-\Delta_{ND})^{-1}:X^*\to X$$
such that for every $\zeta\in X^*$, the element $\psi_\zeta:=(-\Delta_{ND})^{-1}\zeta\in X$ is the unique weak solution of the problem
\begin{equation}\label{problemX}
\left\{
\begin{array}{ll}
\dsy -\Delta\psi_\zeta = \zeta,&\quad\text{in $\Omega$}, \\
\psi_\zeta = 0,&\quad\text{in $B_D$}, \\
\dsy\de_\nu\psi_\zeta = 0,&\quad\text{in $B_N$}.
\end{array}\right.
\end{equation}
Thus, $(-\Delta_{ND})^{-1}$ is well-posed and linear. Moreover, it is injective by the uniqueness, and it is continuous by estimate \eqref{eq.stimacontF}, i.e.\,
\begin{equation}\label{eq.stimaDelta}
\|\nabla\psi_\zeta\|_0=\|\psi_\zeta\|_{X}\leq \|\zeta\|_{X^*}.
\end{equation} 
Thus, by the \textit{Open mapping theorem}, $(-\Delta_{ND})^{-1}$ is a topological isomorphism onto its image $\tilde X:=R((-\Delta_{ND})^{-1})\,\, (\subseteq X)$ with inverse map $-\Delta|_{\tilde X}:\tilde X\to X^*,$ where we have endowed the space $\tilde X$ with the norm $\|\cdot\|_{\tilde X}:=\|\nabla\cdot\|_0$.

\begin{prop}
$\tilde X$ is an Hilbert space endowed with the inner product $(\nabla\cdot,\nabla\cdot)_0$.
\end{prop}
\begin{proof}
If $(\psi_k)_{k\in\N}$ is a Cauchy sequence in $\tilde X$, then $\psi_k\to\psi$ in $X$ with respect to $\|\nabla\cdot\|_0$ with $\psi\in X$. Thus, $\psi_k\rightharpoonup\psi$ in $X$, implying that (recall that the functional $(\nabla\cdot,\nabla\varphi)_0\in X^*$)
\begin{equation}\label{eqweakconv}
(\nabla\psi_k,\nabla\varphi)_0\xrightarrow{k\to+\infty}(\nabla\psi,\nabla\varphi)_0\quad\text{for every $\varphi\in X$}.
\end{equation}
Since $(\psi_k)_{k\in\N}\subseteq\tilde X=R((-\Delta_{ND})^{-1})$, there exists a sequence $(\zeta_k)_{k\in\N}\subseteq X^*$ such that
$$(\nabla\psi_k,\nabla\varphi)_0=\langle\zeta_k|\varphi\rangle_{X^*,X}\quad\text{for every $\varphi\in X$}.$$
Moreover,
\begin{align*}
\|\zeta_k\|_{X^*}&=\sup_{0\neq \varphi\in X,\|\varphi\|_{X}\leq1 }|\langle\zeta_k|\varphi\rangle_{X^*,X}|=\sup_{0\neq \varphi\in X,\|\varphi\|_{X}\leq1 }|(\nabla\psi_k,\nabla\varphi)_0|\\
&\leq\sup_{0\neq \varphi\in X,\|\varphi\|_{X}\leq1 } \|\nabla\psi_k\|_0\|\varphi\|_X\leq \|\nabla\psi\|_0+1,\quad\text{for $k\gg 1$},
\end{align*}
showing that $(\zeta_k)_{k\in\N}$ is bounded in $X^*$. 
By \textit{Banach-Alaoglu theorem}, up to pass to a subsequence, $\zeta_k\rightharpoonup\zeta$ in $X^*$, i.e.\,for every $\varphi\in X^{**}=X$ ($X$ is reflexive since it is an Hilbert space), 
$$\langle\zeta_k|\varphi\rangle_{X^*,X}=\langle\varphi|\zeta_k\rangle_{X^{**},X^*}\xrightarrow{k\to+\infty}\langle\varphi|\zeta\rangle_{X^{**},X^*}=\langle\zeta|\varphi\rangle_{X^*,X}.$$
Combining previous computations with \eqref{eqweakconv}, we obtain $\psi=(-\Delta_{ND})^{-1}\zeta$, i.e.\,$\psi\in \tilde X$, showing that $\tilde X$ is an Hilbert space with respect to $\|\nabla\cdot\|_0$.
\end{proof}\medskip

We observe that, by estimate \eqref{eq.stimaDelta}, we have
\begin{align*}
\|(-\Delta_{ND})^{-1}\|_{\mathcal{L}(X^*,\tilde X)}&=\sup_{0\neq \zeta\in X^*,\|\zeta\|_{X^*}\leq1 } \|(-\Delta_{ND})^{-1}\zeta\|_{\tilde X}\\
&\leq\sup_{0\neq \zeta\in X^*,\|\zeta\|_{X^*}\leq1 } \|\zeta\|_{X^*} = 1.
\end{align*}
Moreover,
\begin{align*}
\|-\Delta|_{\tilde X}\|_{\mathcal{L}(\tilde X,X^*)}&=\sup_{0\neq \phi\in \tilde X,\,\|\phi\|_{\tilde X}\leq1 }\left(\sup_{0\neq \varphi\in X,\,\|\varphi\|_{X}\leq1 }|\langle-\Delta\phi|\varphi\rangle|\right)\\
&\leq \sup_{0\neq \phi\in \tilde X,\,\|\phi\|_{\tilde X}\leq1 }\left(\sup_{0\neq \varphi\in X,\,\|\varphi\|_{X}\leq1 }|(\nabla\phi,\nabla\varphi)_0|\right)\leq 1.
\end{align*}

\bigskip

\section*{Acknowledgements}

The author is supported by the University of Bologna, funds for selected research topics, and by GNAMPA of INdAM (Istituto Nazionale di Alta Matematica ``F. Severi''), Italy.

I would like to thank INFN-Firenze, and in particular Lucio Anderlini, Alessandro Bombini, and Clarissa Buti, for introducing me to this problem and for many valuable discussions on its physical aspects.

I am deeply grateful to Simone Ciani and Annalisa Baldi for their many insightful suggestions regarding the development and organization of this work. I also sincerely thank Felix Brandt for carefully reading a preliminary version of the manuscript and for his valuable comments and suggestions concerning its presentation.

\bibliographystyle{amsplain}

\bibliography{R_submitted}

\bigskip
\tiny{
\noindent
Alessandro Rosa 
\par\noindent
Universit\`a di Bologna, Dipartimento
di Matematica\par\noindent Piazza di
Porta S.~Donato 5, 40126 Bologna, Italy.
\par\noindent
e-mail:
alessandro.rosa15@unibo.it 
}

\end{document}